\RequirePackage{rotating}
\documentclass[12pt]{amsart}
\usepackage[graphicx]{realboxes}
\usepackage{float}
\restylefloat{table}
\usepackage{placeins}

\usepackage{amsmath}
\usepackage{amssymb}
\usepackage{epsfig}
\usepackage{wasysym}
\usepackage{graphicx}
\usepackage{tikz}
\usepackage{tikz-cd}
\usepackage{bm}
\usepackage{xcolor}
\usepackage{listings}
\numberwithin{equation}{section}
\usepackage{extpfeil}
\usepackage{array}
\usepackage{adjustbox}
\usepackage{longtable}
\usepackage{makecell}
\usepackage[all]{xy}
\usepackage{longtable}
\usepackage{rotating}
\input xy
\xyoption{all}
\usepackage{multirow}

\usetikzlibrary{positioning}
\usepackage[colorlinks=false,urlbordercolor=white]{hyperref}
  
\tikzset{sgplattice/.style={inner sep=1pt,norm/.style={red!50!blue},char/.style={blue!50!black},
  lin/.style={black!50}},cnj/.style={black!50,yshift=-2.5pt,left=-1pt of #1,scale=0.5,fill=white}}

\DeclareFontFamily{U}{mathb}{\hyphenchar\font45}
\DeclareFontShape{U}{mathb}{m}{n}{
      <5> <6> <7> <8> <9> <10> gen * mathb
      <10.95> mathb10 <12> <14.4> <17.28> <20.74> <24.88> mathb12
      }{}
\DeclareSymbolFont{mathb}{U}{mathb}{m}{n}
\DeclareMathSymbol{\righttoleftarrow}{3}{mathb}{"FD}

\calclayout
\allowdisplaybreaks[3]

\theoremstyle{plain}
\newtheorem{prop}{Proposition}[section]

\newtheorem{theo}[prop]{Theorem}
\newtheorem{coro}[prop]{Corollary}

\newtheorem{lemm}[prop]{Lemma}

\theoremstyle{definition}

\newtheorem{rema}[prop]{Remark}

\newtheorem{exam}[prop]{Example}

\newcommand{\actsfromleft}{\mathrel{\reflectbox{$\righttoleftarrow$}}}
\newcommand{\actsfromright}{\righttoleftarrow}

\def\cN{{\mathcal N}}
\def\cO{{\mathcal O}}

\def\cX{{\mathcal X}}

\def\sA{{\mathsf A}}
\def\sD{{\mathsf D}}
\def\sE{{\mathsf E}}
\def\fA{{\mathfrak A}}

\def\fD{{\mathfrak D}}

\def\fS{{\mathfrak S}}

\def\fS{{\mathfrak S}}

\def\bA{{\mathbb A}}
\def\bG{{\mathbb G}}
\def\bP{{\mathbb P}}

\def\bZ{{\mathbb Z}}

\def\bC{{\mathbb C}}

\def\rH{{\mathrm H}}

\renewcommand{\thefootnote}{\fnsymbol{footnote}}

\def\Bl{\mathrm{Bl}}
\def\Pic{\mathrm{Pic}}

\def\Aut{\mathrm{Aut}}

\def\GL{\mathsf{GL}}

\def\PGL{\mathsf{PGL}}

\def\Burn{\mathrm{Burn}}

\def\lim{\mathrm{lim}}

\def\IJ{\mathrm{IJ}}
\def\JJ{\mathrm{J}}

\def\rank{\mathrm{rank}}
\def\Sing{\mathrm{Sing}}

\def\Cl{\mathrm{Cl}}

\newcommand{\ZZ}[1]{{\color{cyan} \sf $\bullet$ Z: [#1]}}

\makeatother
\makeatletter

\begin{document}
\title[Equivariant geometry of cubic threefolds]{Equivariant geometry of singular cubic threefolds, II}

\author[I. Cheltsov]{Ivan Cheltsov}
\address{Department of Mathematics, University of Edinburgh, UK}

\email{I.Cheltsov@ed.ac.uk}

\author[L. Marquand]{Lisa Marquand}
\address{
  Courant Institute,
  251 Mercer Street,
  New York, NY 10012, USA
}

\email{lisa.marquand@nyu.edu}

\author[Y. Tschinkel]{Yuri Tschinkel}
\address{
  Courant Institute,
  251 Mercer Street,
  New York, NY 10012, USA
}

\email{tschinkel@cims.nyu.edu}

\address{Simons Foundation\\
160 Fifth Avenue\\
New York, NY 10010\\
USA}

\author[Zh. Zhang]{Zhijia Zhang}

\address{
Courant Institute,
  251 Mercer Street,
  New York, NY 10012, USA
}

\email{zhijia.zhang@cims.nyu.edu}

\date{\today}

\begin{abstract}
We study linearizability of actions of finite groups on cubic threefolds with nonnodal isolated singularities.  
\end{abstract}

\maketitle
\thefootnote{2020 Mathematics Subject Classification: 14E07, 14L30}
\section{Introduction}
\label{sect:intro}
Among the central problems in birational geometry is the {\em linearizability} problem, as
well as the closely related 
{\em rationality problem}.  
The first is about identifying regular actions of finite groups $G$ on algebraic varieties which are equivariantly birational to actions of $G$ on $\bP(V)$, where $V$ is a representation of $G$. The second could be viewed as a special case, when $G$ is the trivial group, i.e., identifying varieties birational to projective space. These questions can be asked over the complex numbers $\bC$, or arbitrary ground fields. One of the distinguishing features of this research is the rich interplay between arithmetic and geometric aspects. 

In this paper, we focus on {\em linearizability} and {\em stable linerizability} of actions of finite groups on singular 
cubic threefolds $X\subset \bP^4$, over an uncountable algebraically closed field $k$ of characteristic zero; recall that a  $G$-action on $X$ is stably linearizable if the action on $X\times \bP^n$ is linearizable, with $G$ acting trivially on the second factor. 

We extend our investigations of the {\em nodal} case
in \cite{CTZ}, \cite{CTZcub} to cover the remaining cases of isolated singularities. 
We rely on the recent classification of such singularities in \cite{viktorova}. 
In detail, we only consider situations when the $G$-action does not fix one of the singular points, since in that case, the $G$-action is linearizable via projection from this point. Under this assumption, in the nonnodal case, 
there are at most 6 singular points, all of which are necessarily $\sA\sD\mathsf E$ singularities. The linear position of the singularities affects the possible automorphism groups $G\subset \Aut(X)$ - we use the {\em defect}
$$
d=d(X):=\rank\, \Cl(X)-1,
$$ 
where $\Cl(X)$ is the class group of $X$, to distinguish some cases. 
Going through the list of configurations of isolated singularities in
\cite[Table 7, 8 and 9]{viktorova} we extract all nonnodal cases that are not {\em a priori} 
linearizable. We compute, in Section~\ref{sect:defect}, the defect, using \cite[Theorem 1.1]{MV}:
\begin{itemize}
\item $m\sA_1$, $m=2,\ldots, 10$, handled in \cite{CTZ}, \cite{CTZcub},
\item $m\sA_2$, $m= 2,3, 4, 5$, $d=0$,
\item $2\sA_2+m\sA_1$, $m=2,3,4$, and 
$d=\begin{cases} 1 &\text{ if  } m=4, \\
 0 & \text{ otherwise, }
\end{cases}$
\item $2\sA_3+m\sA_1$, $m=2,3,4$, and $d=\begin{cases}
    3 & \text{ if } m=4,\\
   2 & \text{ if } m=3,\\
    1 \text{ or } 2 & \text{ if } m=2,
\end{cases}$
\item $2\sA_3$ and $d=0$ or $1$,
\item $2\sA_4$ and $d=0$,
\item $2\sA_5$ and $d=1$,
\item $3\sA_2+2\sA_1$, $d=0$,
\item $3\sA_3$, $d=1 $ or $2$,
    \item $2\sD_4$, $d=2$,
    \item $2\sD_4+2\sA_1$, $d=3$,    
    \item $2\sD_4+3\sA_1$, $d=4$,
    \item $3\sD_4$, $d=4$, there is a unique such threefold \cite[Theorem 3.2]{ACT}.
\end{itemize}
Note that in each of these cases the cubic $X$ is GIT semistable \cite{ACT}.

Starting from here, the strategy is transparent: 
describe {\em normal forms} of the cubics for each relevant singularity type, compute the full automorphism group $\Aut(X)$, deploy the known obstructions to linearizability, such as

\begin{itemize}
    \item[]{\bf (H1)} cohomology of the $G$-action on the Picard group $\Pic(\tilde{X})$, for a smooth model $\tilde{X}$ of $X$, 
    \item[]{\bf (IJ)} equivariant intermediate Jacobians, see \cite[Section 2]{CTZcub},
     \item[]{\bf (Burn)} Burnside invariants \cite{BnG}, 
    \item[]{\bf (Sp)} equivariant specialization, 
\end{itemize}
to identify nonlinearizable actions. While the nonvanishing of the {\bf (H1)}, {\bf (IJ)}, or {\bf (Burn)} obstructions exclude linearizability of the given threefold, the specialization technique only yields information for a {\em very general} member of the corresponding family (which explains our restriction to an uncountable ground field $k$). In practice, it is very difficult to obtain a result for {\em every} member; this is well-understood in the study of rationality. 
In the remaining cases, we look for linearizability constructions. The implementation of this strategy is quite involved, and relies on extensive use of {\tt magma}.  

In applications of equivariant specialization {\bf (Sp)}, we need
detailed information about degenerations of singularities. Recall that a (combination of) $\sA\sD\sE$ singularities $T$ degenerates to $T'$ if and only if the Dynkin diagram of the root system of $T$ is an induced subgraph of the Dynkin diagram of the root system of $T'$ (see \cite[Section 5.9]{Arnold}). We record the possible degenerations of singularities of cubic threefolds:

{\tiny
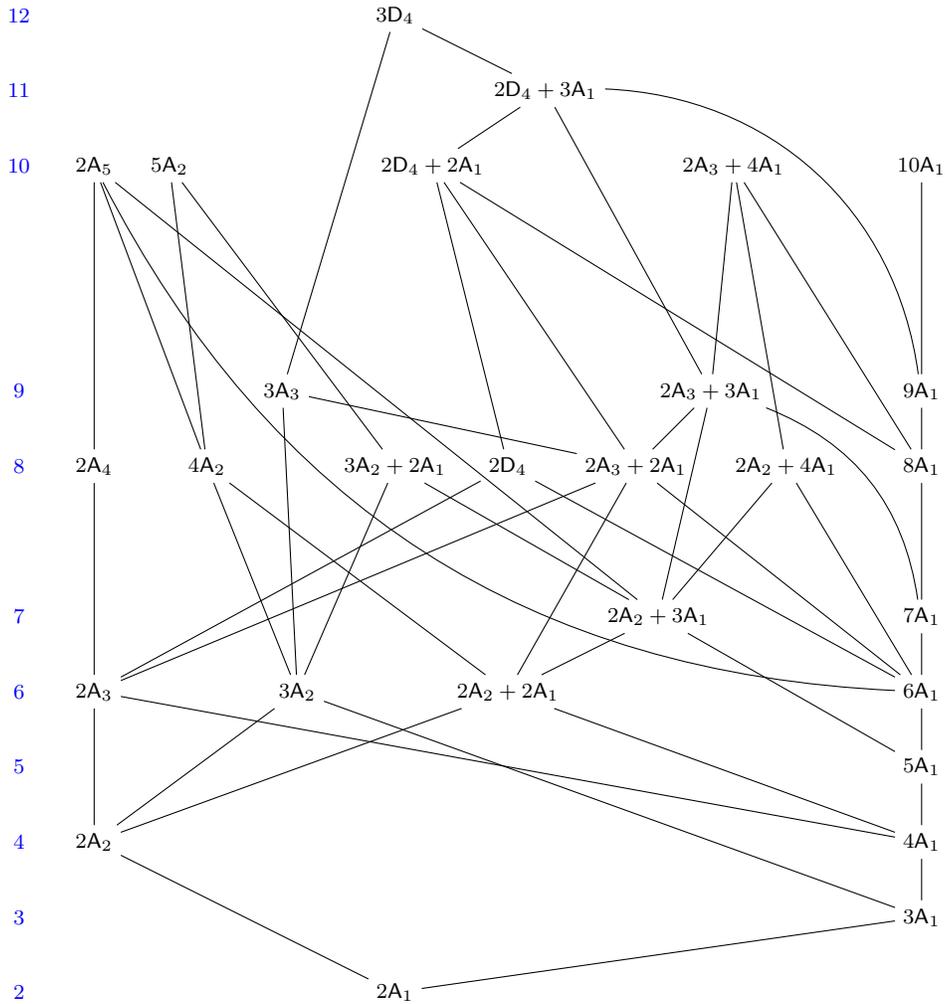
\[
\begin{tikzpicture}[commutative diagrams/every diagram]
 \node(l12) at (-1,7)
 {$\color{blue} 12$};
  \node (28) at (4,7)
	{$3\sD_4$};

 \node(l11) at (-1,6)
 {$\color{blue} 11$};
 \node (27) at (6,6)
	{$2\sD_4+3\sA_1$};

 \node(l10) at (-1,5)
 {$\color{blue} 10$};
\node (26) at (11,5)
	{$10\sA_1$};
\node (23) at (0,5)
	{$2\sA_5$};
\node (25) at (1,5)
	{$5\sA_2$};
\node (22) at (4.5,5)
	{$2\sD_4+2\sA_1$};
 \node (24) at (8.5,5)
	{$2\sA_3+4\sA_1$};

 \node(l9) at (-1,2)
 {$\color{blue} 9$};
\node (19) at (2.5,2)
	{$3\sA_3$};
\node (20) at (8.2,2)
	{$2\sA_3+3\sA_1$};
\node (21) at (11,2)
	{$9\sA_1$};

 \node(l8) at (-1,1)
 {$\color{blue} 8$};
	\node (14) at (0,1)
	{$2\sA_4$};
	\node (18) at (1.5,1)
	{$4\sA_2$};
	\node (13) at (4,1)
	{$3\sA_2+2\sA_1$};
	\node (12) at (5.5,1)
	{$2\sD_4$};
	\node (15) at (7.2,1)
	{$2\sA_3+2\sA_1$};
	\node (16) at (9.2,1)
	{$2\sA_2+4\sA_1$};
	\node (17) at (11,1)
	{$8\sA_1$};

 \node(l7) at (-1,-1)
 {$\color{blue} 7$};
\node (10) at (7.5,-1)
	{$2\sA_2+3\sA_1$};
\node (11) at (11,-1)
	{$7\sA_1$};

 \node(l6) at (-1,-2)
 {$\color{blue} 6$};
\node (6) at (0,-2)
	{$2\sA_3$};
\node (8) at (2.7,-2)
	{$3\sA_2$};
\node (7) at (5.5,-2)
	{$2\sA_2+2\sA_1$};
\node (9) at (11,-2)
	{$6\sA_1$};

 \node(l5) at (-1,-3)
 {$\color{blue} 5$};
\node (5) at (11,-3)
	{$5\sA_1$};

 \node(l4) at (-1,-4)
 {$\color{blue} 4$};
\node (4) at (11,-4)
	{$4\sA_1$};
 \node (3) at (0,-4)
	{$2\sA_2$};

 \node(l3) at (-1,-5)
 {$\color{blue} 3$};
\node (2) at (11,-5)
	{$3\sA_1$};

 \node(l2) at (-1,-6)
 {$\color{blue} 2$};
\node (1) at (4,-6)
	{$2\sA_1$};

	\path[-]
       (2) edge (1)
(3) edge (1)
(4) edge (2)
(5) edge (4)
(6) edge (3)
(6) edge (4)
(7) edge (3)
(7) edge (4)
(8) edge (2)
(8) edge (3)
(9) edge (5)
(10) edge (5)
(10) edge (7)
(11) edge (9)
(12) edge (6)
(12) edge (9)
(13) edge (8)
(13) edge (10)
(14) edge (6)
(15) edge (6)
(15) edge (7)
(15) edge (9)
(16) edge (9)
(16) edge (10)
(17) edge (11)
(18) edge (7)
(18) edge (8)
(19) edge (8)
(19) edge (15)
(20) edge (10)
(20) edge[bend left] (11)
(20) edge (15)
(21) edge (17)
(22) edge (12)
(22) edge (15)
(22) edge (17)
(23) edge[bend right] (9)
(23) edge (10)
(23) edge (14)
(23) edge (18)
(24) edge (16)
(24) edge (17)
(24) edge (20)
(25) edge (13)
(25) edge (18)
(26) edge (21)
(27) edge (20)
(27) edge[bend left=40] (21)
(27) edge (22)
(28) edge (19)
(28) edge (27);
\end{tikzpicture}
\]
}

\

For a given $G$-action on a nonnodal cubic $X$, 
{\bf (H1)} obstruction does not vanish only in the following cases: 
$$
2\sA_5, \quad 2\sD_4+2\sA_1\quad\text{ and }\quad 3\sD_4.
$$
In each of these three cases, the full automorphism group $\mathrm{Aut}(X)$ is infinite,
and the vanishing of the {\bf (H1)} obstruction is equivalent to the linearization of the $G$-action, 
see Proposition~\ref{prop:h1}.

\

We proceed to summarize the results: $X$ is a nonnodal cubic threefold, with singularities as above, and $G\subseteq \Aut(X)$ a finite group.

\subsection*{Two singularities}
\begin{itemize}
    \item $2\sA_2$: the $G$-action is linearizable if and only if $G$ fixes a singular point.
    \item $2\sA_3$: 
    \begin{itemize}
        \item $d(X)=0$: if $X$ is very general, the $G$-action is not linearizable, with a possible exception when $\Aut(X)=\fD_4$ and $G=C_4$, see Corollary~\ref{coro:2A3noplanelin}, 
        \item $d(X)=1$: 
        $X$ is $G$-equivariantly birational to a smooth intersection of two quadrics $X_{2,2}\subset \bP^5$; the $G$-action is linearizable if and only if there is a $G$-stable line on $X_{2,2}$, by \cite{HT-intersect}.  
    \end{itemize}
    \item $2\sA_4$: if $X$ is very general, the $G$-action is not linearizable, with a possible exception when $\Aut(X)=C_6$ and $G=C_2$, see Proposition~\ref{prop:2A4spec2A5}.
    \item $2\sA_5$:  the $G$-action is linearizable if and only if the {\bf (H1)} obstruction vanishes, which happens if and only if $G$ acts trivially on the class group $\Cl(X)\simeq \bZ^2$, see Corollary~\ref{coro:2a5sum}.    
    \item $2\sD_4$:  {\bf (Burn)} and {\bf (Sp)} settle the linearizability problem for most actions.  
\end{itemize}

\subsection*{Three singularities} 
\begin{itemize}
\item $3\sA_2$  and  $3\sA_3$: we expect that the $G$-action is linearizable if and only if $G$ fixes a singular point, and  we confirm this in many cases using {\bf (Burn)} and {\bf (Sp)}, to cohomology for a specific $C_3$-action.   
In the $3\sA_2$ case, the intermediate Jacobian $\IJ(\tilde{X})$ of a minimal resolution of singularities $\tilde{X}\to X$ is the Jacobian of a smooth curve of genus 2, and the intermediate Jacobian obstruction of \cite{CTT} may be applicable. 
\item $3\sD_4$: the $G$-action is linearizable if and only if the {\bf (H1)} obstruction vanishes, see  Proposition~\ref{prop:3D4coho}.
 \end{itemize}

\subsection*{Four singularities}

Many $G$-actions are nonlinearizable, via {\bf (Burn)}, see Proposition~\ref{prop:burn-ob-4}. 
\begin{itemize}
    \item $2\sA_2+2\sA_1$ and  $4\sA_2$: such $X$ are equivariantly birational to a smooth divisor of degree $(1,1,1,1)$ in $(\bP^1)^4$, we expect that the action is linearizable if and only if $G$ fixes a singular point; we prove this for very general $X$ in Proposition~\ref{prop:2a22a1spec2d42a1}, 
    respectively, in Proposition~\ref{eq:4A2->2A5}. 
    \item $2\sA_3+2\sA_1$:
    \begin{itemize}
        \item $d(X)=1$: the $G$-action is linearizable, by Lemma~\ref{lemm:d1}. 
        \item $d(X)=2$: the $G$-action on very general $X$ is linearizable if and only if it fixes a singular point, by Proposition~\ref{prop:2a32a1spec2d42a1}. 
    \end{itemize} 
    \item $2\sD_4+2\sA_1$:  the $G$-action is linearizable if and only if the {\bf (H1)} obstruction vanishes, by Corollary~\ref{coro:2D4+2A1}. 
  \end{itemize}

\subsection*{Five singularities}
\begin{itemize}
    \item $ 2\sA_2+3\sA_1, 2\sA_3+3\sA_1, 3\sA_2+2\sA_1, 5\sA_2$:  the $G$-action is linearizable.
    \item $2\sD_4+3\sA_1$: there is a unique such threefold, with infinite automorphisms, $G$-equivariantly birational to a smooth quadric without fixed points; {\bf (Burn)} obstructs some of the actions, e.g., for $G=C_2^2\times \fS_3$. The linearizability problem for smooth quadric threefolds is still open. 
    \end{itemize}

\subsection*{Six singularities}

All $G$-actions are linearizable.

\

Here is the roadmap of the paper: 
In Section~\ref{sect:defect}, we compute the {\em defect} $d(X)$, in all cases. In Section~\ref{sect:aut}, we explain how to compute the automorphism group $\Aut(X)$, and implement the algorithm in an example. Section~\ref{sect:coho} is devoted to computations of the Picard group of a minimal resolution of singularities $\tilde{X}$ of $X$ and of group cohomology
$\rH^1(G, \Pic(\tilde{X}))$, for subgroups $G\subseteq \Aut(X)$; the nonvanishing of this invariant is an obstruction to linearizability. The subsequent sections are organized by the number of singular points.

\

\noindent
{\bf Acknowledgments:} 
The first author was partially supported by the Leverhulme Trust grant RPG-2021-229,
EPSRC grant EP/Y033485/1, and Simons Collaboration grant \emph{Moduli of Varieties}.
The third author was partially supported by NSF grant 2301983.


\section{The Defect}
\label{sect:defect}

In this section, we compute the defect 
$$
d(X):=\mathrm{rk} \,\Cl(X) -1,
$$
where $\Cl(X)$ is the class group, 
for cubic threefolds $X$ with specified combinations of singularities, using the projection method. The proofs follow closely those in \cite[\S 4]{viktorova}, however, we feel that the presentation will be useful for the reader. We always project from the worst singularity $q\in \Sing(X)$, the singular locus of $X$.

\subsection*{Projection method}\label{sec: proj}
We review the projection method outlined in \cite{MV} (see also \cite{viktorova}): fix 
$q\in \Sing(X)$ 
and choose coordinates so that $q=[1:0:0:0:0]$. Then $X$ is given by
\begin{equation}\label{cubic eqn}
	x_1f_2(x_2,\dots, x_5)+f_3(x_2,\dots, x_5)=0,
\end{equation} 
where $f_2, f_3$ are homogeneous, of degree 2 and 3, respectively. 
Projection from $q$ gives a birational map $X\dashrightarrow \bP^3$, factoring as
\[\begin{tikzcd}
	& \Bl_qX \cong \Bl_{C_q}\bP^3\arrow[dl]\arrow[dr, "\phi"] &\\
	X\arrow[rr, dashrightarrow] & & \bP^3
\end{tikzcd}
\]
This yields
\begin{align*}
    Q_q:=\{f_2(x_2,\dots, x_5)=0\}\subset \bP^3,\\
    S_q:=\{f_3(x_2,\dots, x_5)=0\}\subset \bP^3,\\
    C_q:= Q_q\cap S_q\subset \bP^3.
\end{align*}
The curve $C_q$ parameterizes lines contained in $X$ passing through $q$. Recall that by \cite[Theorem 1.2]{Wall99}, the singularities of $C_q$ correspond to that of $X$ away from $q$.

\begin{theo}\cite[Theorem 1.2]{Wall99}\label{thm: Wall}
    Consider a singular point $p\in C_q$, and assume $S_q$ is smooth.
    \begin{enumerate}
        \item If $Q_q$ is smooth at $p$, then $X$ has a unique singular point on the line $\langle p, q \rangle$ other than $q$, and the singularity has the same type.
        \item If $Q_q$ is singular at $p$, then the only singular point of $X$ on $\langle p, q \rangle$ is $q$, and the blow-up $\Bl_qX$ has a singular point of the same type as $p\in C$ on $\phi^{-1}(q)|_E$ for $\phi$ as above.
    \end{enumerate}
\end{theo}
\begin{lemm}\label{lemm: defect formula}
    Let $X$ be a cubic threefold with singularities as above. Then:
    \begin{enumerate}
        \item If $q$ is a $\sD_4$-point, then $Q_q$ is the union of two planes and 
        $$d(X)=(\# \text{ irreducible components of }C_q)-2;$$
        \item If $q$ is an $\sA_n$-point with $n\geq 2,$ then $Q_q$ is a quadric cone and $$d(X)=(\# \text{ irreducible components of }C_q)-1.$$ 
        \item If $q$ is an $\sA_n$-point with $n\geq 3,$ then in addition $C_q$ passes through the cone point of $Q_q$  and has an $\sA_{n-2}$-singularity there.
    \end{enumerate}
\end{lemm}

\begin{proof}
    The statements about $Q_q$ are a combination of \cite[Claim A.10, A.11]{viktorova}. The formula for the defect is \cite[Theorem 1.2]{MV}.
\end{proof}


\subsection*{Projection from $\sA_2$}

\begin{lemm}\label{lem: defect mA_2}
    Let $X$ be a cubic threefold. Then $d(X)=0$, when $X$ has
    \begin{itemize}
        \item $m\sA_2$-singularities, for $1\leq m\leq 5$, or
        \item $(2\sA_2 +m\sA_1)$-singularities, for $1\leq m\leq 3$, or
        \item $(3\sA_2+2\sA_1)$-singularities.
    \end{itemize}
\end{lemm}

\begin{proof}
   Projecting from $q\in \Sing(X),$ we see that $C_q$ must be irreducible, otherwise $X$ would have at least $4\sA_1$-singularities. It follows that $C_q$ is an irreducible $(2,3)$ complete intersection curve, and so the defect is 0 by Lemma \ref{lemm: defect formula} (2), see \cite[Proposition 4.9]{viktorova}.
\end{proof}

\begin{prop}\label{prop:4A1+2A2}
Let $X$ be a cubic with $(4\sA_1+2\sA_2)$-singularities. 
Then $d(X)=1$, and 
\begin{enumerate}
 \item $X$ contains exactly one plane $\Pi$, containing the 4 nodes, 
\item  the line containing the $2\sA_2$-points is disjoint from $\Pi$.
    \end{enumerate}
\end{prop}

\begin{proof}
 Let $q\in\Sing(X)$ be an $\sA_2$-point. By \cite[Proposition 4.9]{viktorova}, we see that $C_q=A\cup B,$ where $A$ is a hyperplane section of the quadric cone $Q_q$ not passing through the cone point, and $B$ is an irreducible curve with an $\sA_2$-singularity. The computation of $d(X)$ follows from Lemma \ref{lemm: defect formula} (2). 

Let $Z\subset \bP^4$ be the cone over $A$ with vertex $q$; then $Z\subset X$, and $Z$ spans a hyperplane, which intersects $X$ in $Z\cup \Pi$, with $\Pi\cap Z=A$ containing the four nodes of $X$.

    For the second claim, let $L\subset X\subset \bP^4$ be the line between the two $\sA_2$-points. Then $L$ intersects the hyperplane spanned by $Z$ in exactly one point $q$, and it follows that $L\cap \Pi=\varnothing$.
\end{proof}

\subsection*{Projection from $\sA_3$}

\begin{lemm}
    Let $X$ be a cubic with $2\sA_3$-singularities. Then either
    \begin{enumerate}
        \item $d(X)=0$, or
        \item $d(X)=1$, and there is a unique plane containing the $\sA_3$-points.
    \end{enumerate}
\end{lemm}
\begin{proof}
    We project from an $\sA_3$-point $q$. The curve $C_q\subset Q_q$ has an $\sA_1$-singularity at the cone point of $Q_q$. Since $X$ has a second $\sA_3$-point, so must $C_q$. By \cite[Lemma 4.4, Proposition 4.8]{viktorova}, this implies that either:
    \begin{itemize}
        \item $C_q$ is irreducible with an $\sA_3$-singularity away from the cone point of $Q_q$, or
        \item $C_q=A\cup B,$ where $A$ is a fiber in the ruling of $Q_q,$ and $B$ is a smooth genus $2$ curve tangent to $A$ at a single point, corresponding to the second $\sA_3$-point. By Lemma \ref{lemm: defect formula} (2), we see that $d(X)=1$, and the plane is given by the cone over $A$ with vertex $q$ - it thus contains  the two $\sA_3$-points.
    \end{itemize}
\end{proof}

\begin{lemm}
    Let $X$ be a cubic with $3\sA_3$-singularities. Then either
    \begin{enumerate}
        \item $d(X)=1$ and $\Cl(X)$ is freely generated by two classes of cubic scrolls contained in $X$,
        \item $d(X)=1$ and there is a unique plane $\Pi$ containing exactly two $\sA_3$-points, in particular, there is a distinguished $\sA_3$-point contained in two planes, or 
        \item $d(X)=2$ and there are exactly three planes, each containing exactly two $\sA_3$-points. 
    \end{enumerate}
\end{lemm}
\begin{proof}
    We project from an $\sA_3$-point $q$. 
    The curve $C_q\subset Q_q$ must have an $\sA_1$-singularity at the cone point of $Q_q$. According to \cite[Proposition 4.8]{viktorova} there are four ways of forcing two additional $\sA_3$-singularities:
    \begin{itemize} 
    \item $C_q=A\cup B$, where both $A$ and $B$ are twisted cubics passing through the cone point, and tangent in two other points. By \cite[Lemma 4.4]{MV}, $X$ contains two families of cubic scrolls, that freely generate $\Cl(X)$ (see also \cite{HT-determinant}).
  \item $C_q=A\cup B,$ where $A$ is a hyperplane section of $Q_q$ not passing through the cone point, and $B$ is a smooth curve of genus 1. Further, $A$ is tangent to $B$ at two distinct points with multiplicity 2. In this case, $d(X)=1$ by Lemma \ref{lemm: defect formula} (2): we see the plane as in Proposition~\ref{prop:4A1+2A2}. It contains the remaining two $\sA_3$-points.

  \item  $C_q=A\cup B$, where $A$ is a ruling of $Q_q$ and $B$ a genus 2 curve. Furthermore, $A$ is tangent to $B$ at a point, and $B$ has an additional $\sA_3$-singularity. This is the same arrangement as the previous case, where we are instead projecting from an $\sA_3$-point that is contained in the unique plane.
  
\item $C_q=A_1\cup A_2\cup B$, where $B$ is a smooth curve of genus 1 and each $A_i$ is a distinct line in the ruling of $Q_q$ which is tangent to $B$. In this case again one uses Lemma \ref{lemm: defect formula} (2) to see $d(X)=2$, and the planes are given as the cone over each $A_i$, along with the residual plane from intersecting the hyperplane spanned by the $A_i$. 
\end{itemize}
\end{proof}

\begin{prop}\label{prop:4A1+2A3}
     Let $X$ be a cubic with $2\sA_3+4\sA_1$-singularities. Then
    \begin{enumerate}
        \item $d(X)=3$, and the extra generators of $\Cl(X)$ are planes,
        \item there is a unique plane $\Pi\subset X$ containing all four nodes,
        \item each $\sA_3$-point is contained in two planes, each containing two other nodes,
        \item the line containing the $2\sA_3$-points is disjoint from $\Pi$.
    \end{enumerate}
\end{prop}

\begin{proof}
    We project from an $\sA_3$-point $q\in X$. By \cite[Proposition 4.8, corrected version]{viktorova}, we see that $C_q=A_1\cup A_2\cup B_1\cup B_2$, where $A_1, A_2$ are two distinct lines in the ruling of the quadric cone $Q_q$, and $B_1, B_2$ are two hyperplane sections of $Q_q$ not passing through the cone point and tangent to each other at $p\in Q_q$.
    The computation of $d(X)$ follows from Lemma \ref{lemm: defect formula}.

    We see that there are two planes containing $q$; namely, the cones over $A_{1}, A_{2}$ with vertex $q$. Each contains two nodes of $X$: indeed, $C_q$ has a node at each of the points $p_{ij}\in A_i\cap B_j$, and thus by,  Theorem \ref{thm: Wall}, there is a node on the line $\langle q, p_i\rangle$.
    Note that the plane $\Pi$ spanned by $A_1, A_2$ contains the four nodes of $X$, and hence is contained in $X$ itself.

Finally, let $L$ denote the line through the two $\sA_3$-points of $X$. We claim $L\cap \Pi=\varnothing$.
Indeed, suppose that $L\cap \Pi\ne\varnothing$.
Then there exists a hyperplane section $F\subset X$ that contains $L$ and $\Pi$. Note that $F$ must split as a union of $\Pi$ and two other planes such that both of them contain $L$, and each of them contains two nodes of $X$. This is impossible: if a plane in $X$ contains four singular points of $X$, one can write down the equation of $X$ and a local computation shows that these singular points must be nodes. 
\end{proof}

\begin{lemm}\label{lemm:2A3mA1defect}
    Let $X$ be a cubic with $2\sA_3+m\sA_1$-singularities, for $2\leq m\leq 3. $ Then
    \begin{enumerate}
        \item When $m=2$, then either 
        \begin{itemize}
            \item $d(X)=1$ and there is a unique plane in $X$ that contains the two $\sA_1$-points and exactly one $\sA_3$-point, 
            \item $d(X)=1$ and there is a unique plane in $X$ containing the two $\sA_3$-points and no other singularities, or
            \item $d(X)=2$ and there are three planes contained in $X$; one plane contains the two $\sA_3$-points, and the other planes each contain one of the $\sA_3$-points and the two $\sA_1$-points.
        \end{itemize}
        \item When $m=3$, then $d(X)=2$ and there are exactly two planes contained in $X$; both planes contain an $\sA_3$-point and exactly $2\sA_1$-points.
    \end{enumerate}
\end{lemm}

\begin{proof}
    Let $q\in \Sing(X)$ be an $\sA_3$-point. By \cite[Proposition 4.8, corrected version]{viktorova}, $C_q$ is necessarily reducible. If $X$ has $2\sA_3+2\sA_1$-singularities, then there are two possibilities:
    \begin{itemize}
        \item $C_q=A\cup B$, with $A$ and $B$ irreducible and intersecting transversely at the cone point. Thus $d(X)=1,$ and the additional class in $\Cl(X)$ is a plane. There are three options: 
    \begin{itemize}
        \item $A$ is a hyperplane section of the quadric, and $B$ is tangent to $A$ at one point, and intersects transversely in two distinct points. The plane is the residual from the cone over $A$, and contains the two $\sA_1$-points and one $\sA_3$-point.
        \item $A$ is a ruling of the quadric, intersecting $B$ in two distinct points; $B$ has an additional $\sA_3$-singularity. The plane is the cone over $A,$ and contains the two $\sA_1$-points and one $\sA_3$-point.
        \item $A$ is a ruling of the quadric, and is tangent to $B$ at a single point; $B$ has two additional $\sA_1$-singularities.
    \end{itemize}

\item $C_q=A\cup B_1\cup B_2,$ and each irreducible. Here, $B_1$ and $B_2$ are distinct rulings of the quadric cone, and $B_1$ is tangent to $A$ in a unique point, whereas $B_2$ intersects $A$ in two distinct points. We see that $d(X)=2,$ and the cone over $B_1$ is a plane containing the two $\sA_3$-points, whereas the cone over $B_2$ is a plane containing only $q$ and the two $\sA_1$-points. Taking the plane spanned by $B_1$ and $B_2$ gives a third plane containing one $\sA_3$-point and two $\sA_1$-points.
\end{itemize}

    Next, suppose that $X$ has $2\sA_3+3\sA_1$-singularities. Then $C_q=A\cup B_1\cup B_2,$ where $A$ is a ruling of the quadric, $B_1$ is a hyperplane section of the quadric, and $B_2$ is a twisted cubic. Further, $B_1$ and $B_2$ are tangent in one point, and intersect transversely in one other point. We see that $d(X)=2,$ and there are exactly two planes in $X$: the first is the cone over $A$, and contains one $\sA_3$-point and $2\sA_1$-points. The second plane is residual to the cone over $B_1,$ and contains one $\sA_3$-point and $2\sA_1$-points. Note that there is one $\sA_1$-point that belongs to both planes, namely the intersection of $A$ and $B_1.$
\end{proof}

\subsection*{Projection from $\sA_4$}

\begin{lemm}
    Let $X$ be a cubic threefold with $2\sA_4$-singularities. Then $d(X)=0.$
\end{lemm}
\begin{proof}
This is case (4) of \cite[Proposition 4.6]{viktorova}: $C_q$ must be irreducible with one $\sA_4$-point  (and an $\sA_2$-point at the cone point of $Q_q$); indeed, having multiple components forces additional singularities on $X$.
\end{proof}

\subsection*{Projection from $\sA_5$}

\begin{lemm}
    Let $X$ be a cubic with $2\sA_5$-singularities. Then $d(X)=1,$ and $\Cl(X)$ is freely generated by two classes of cubic scrolls contained in $X$.
\end{lemm}
\begin{proof}
    We project from one of the $\sA_5$-points $q$. By Theorem \ref{thm: Wall}, the curve $C_q$ necessarily has an $\sA_3$-singularity at the cone point of $Q_q$ - this is impossible if $C_q$ is irreducible. The only possibility for an additional $\sA_5$-singularity is to have $C_q=A\cup B$, where each component is a smooth twisted cubic passing through the cone point of $Q_q$, and intersecting each other in one point with multiplicity $4$. By \cite[Lemma 4.4]{MV}, $X$ contains two families of cubic scrolls, that freely generate $\Cl(X)$ (see also \cite{HT-determinant}).
\end{proof}

\subsection*{Projection from $\sD_4$}
In this case, $Q_q=\Pi_1\cup \Pi_2$, where $\Pi_i\cong \bP^2$, meeting in a line $l$. Note that $C_q=B_1\cup B_2$, where $B_i\subset \Pi_i$ is a cubic curve, and $C_q$ intersects $l$ in three simple points, with each $B_i$ smooth at these points. 

\begin{prop}\label{prop:defectd4pt}
    Let $X$ be a cubic with a $\sD_4$-singularity. Then
    \begin{enumerate}
        \item If $X$ has $2\sD_4$-singularities, then $d(X)=2$ and there are three planes in $X$, each containing the two singular points.
        \item If $X$ has $2\sD_4+2\sA_1$-singularities, then $d(X)=3,$ and there are five planes contained in $X$. Three planes contain both $\sD_4$-points, and each $\sD_4$-point is contained in one other plane containing both nodes.
        \item If $X$ has $2\sD_4+3\sA_1$-singularities, then $d(X)=4$, and there are nine planes contained in $X$. Three planes contain both $\sD_4$-points, and each $\sD_4$-point is contained in three other planes which contain two of the three nodes.
        \item If $X$ has $3\sD_4$-singularities, then $d(X)=4$, and there are nine planes contained in $X$, each containing exactly two singular points.
    \end{enumerate}
\end{prop}
\begin{proof}
    Since $X$ has at least $2\sD_4$-singularities,  at least one of the plane cubics, say $B_1$, is the union of three lines meeting in a point. The cone over each line with vertex $q$ gives a plane in $X$. Consider $span\langle \Pi_1, q\rangle\subset \bP^4$; this hyperplane intersects $X$ in precisely these three planes, giving one relation in $\Cl(X)$ - we see $d(X)=2$ for the case of $2\sD_4$-singularities.

    If $X$ has $r$ additional $\sA_1$-singularities, then $B_2$ must become singular. When $r=2,$ $B_2$ must be a conic and a line, for $r=3$, $B_2$ becomes three lines in a triangle configuration. Each line gives an additional plane in $X$, and the case of three lines gives one relation as before. The defect and plane computation follows. We see that $q$ is contained in three planes which contain the other $\sD_4$-point (corresponding to the cone over $B_1$), and in one plane containing the two nodes. This curve configuration is the only possibility; in particular, if we project from the other $\sD_4$-point, we have the same configuration. The claim follows. 

    If $X$ has $3\sD_4$-singularities, then both $B_1$ and $B_2$ are three lines meeting in a point. The defect is thus $d(X)=4$. Note that there are six planes that contain $q$, each containing one other $\sA_3$-point. By symmetry, there are nine planes contained in $X$.
\end{proof}

\section{Automorphisms}
\label{sect:aut}

In this section, we explain how to classify automorphisms $\Aut(X)$ of singular cubics $X\subset \bP^4$, with $s$ singular points. By convention, the action on the variety is from the right, and the action on the function field from the left. In coordinates, given 
$\mathbf{x}=(x_1,\ldots, x_5)$ and $\sigma \in \Aut(X)$, we put
$$
\sigma({\bf x})=\sigma((x_1,\ldots, x_5))= (x_1,\ldots, x_5) \cdot M_{\sigma},
$$
where $M_{\sigma}\in \GL_5$ is the corresponding matrix.

Throughout, we assume that $\Aut(X)$ does not fix any singular points, 
since otherwise, the action is linearizable. Let $H$ be the maximal subgroup of $\Aut(X)$ which fixes all singular points.
The $\Aut(X)$-action preserves the singular locus, 
yielding an exact sequence
$$
0\to H\to \Aut(X)\stackrel{\rho}{\longrightarrow} \fS_s.
$$
The image of $\rho$ reflects the singularity type, for example, when $X$ has $(2\sA_2+2\sA_1)$-singularities, the image is contained in $\fS_2\times \fS_2$. 

The algorithm for classifying $\Aut(X)$ involves the following steps:
\begin{itemize}
\item 
Find a {\em normal form} of $X$, based on the singularity type. This amounts to fixing appropriate coordinates and simplifying the equation. 
\item 
Determine possible images of $\rho$, and find all lifts to $\Aut(X)$,  depending on parameters in the equation of $X$. 
\item 
For each lift, determine $H$. 
\end{itemize}
Here, we explain the process in a simple case, when $X$ is the unique cubic with 
$3\sD_4$-singularities, see \cite[Theorem 5.4]{All}:

\begin{prop}
\label{prop:autd4}
Let 
\begin{align}\label{eq:3D4eq}
    X:=\{x_1x_2x_3+x_4^3+x_5^3=0\}
\end{align}
be the unique cubic threefold with $3\sD_4$-singularities.
Then
$$
\Aut(X)=\langle\tau_{a,b},\eta,\sigma_{(45)},\sigma_{(123)},\sigma_{(12)}\rangle\simeq(\bG^2_m(k)\times \fS_3)\rtimes \fS_3,
    $$
    where 
    \begin{align*}
        \tau_{a,b} : (\mathbf{x})&\mapsto (ax_1,bx_2,a^{-1}b^{-1}x_3,x_4,x_5),\quad a,b\in k^\times,\\
\eta : (\mathbf{x})&\mapsto (x_1,x_2,x_3,\zeta_3 x_4,\zeta_3^2x_5),\\
         \sigma_{(45)} : (\mathbf{x})&\mapsto (x_1,x_2,x_3, x_5, x_4),\\
         \sigma_{(123)} : (\mathbf{x})& \mapsto (x_3,x_1,x_2, x_4, x_5),\\
         \sigma_{(12)} : (\mathbf{x})&\mapsto (x_2,x_1,x_3, x_4, x_5).
    \end{align*}
In particular, $\eta_1,\eta_2$ generate the first $\fS_3$-factor and $\sigma_{(12)},\sigma_{(123)}$ generate the second $\fS_3$-factor in $\Aut(X)$.
\end{prop}
    
\begin{proof}
Observe that $\sigma_{(12)},\sigma_{(123)}\in\Aut(X)$, and that $\Aut(X)$ acts transitively on the three singular points. It remains to find the subgroup $H\subset\Aut(X)$ fixing all three singular points. 
Based on the form of the equation, we see that 
an $h\in H$ takes the form
    $$
(\mathbf{x})\mapsto(s_1x_1,s_2x_2,s_3x_3,a_1x_4+a_2x_5,a_3x_4+a_4x_5), 
    $$
   for $s_1,s_2,s_3,a_1,a_2,a_3,a_4\in k$. There exists an exact sequence
    $$
    0 \to  H' \to H \stackrel{\psi}{\to} \PGL_2,
    $$
    where $\psi$ is the projection of the $H$-action onto $\bP^1_{x_4,x_5}$, given by the coordinates $x_4, x_5$. 
    The equation of $X$ implies that $\psi(H)$ leaves invariant three points defined by 
    $
    \{x_4^3+x_5^3=0\}\subset \bP^1_{x_4,x_5}.
    $
    The maximal subgroup of $\PGL_2$ leaving these three points invariant is $\fS_3$. To show that $\psi(H)=\fS_3$, one can check that $\eta, \sigma_{(45)}\in H$ and their images in $\PGL_2(k)$ generate $\fS_3$. On the other hand, elements in $\tau\in H'$ are diagonal of the form
    $$
\tau: (\mathbf{x})\mapsto(s_1x_1,s_2x_2,s_3x_3,x_4,x_5).
    $$
    One can check that $s_1s_2s_3=1$ and $\tau$ is given by $\tau_{a,b}$, for  $a,b\in k^\times.$
\end{proof}

\section{Picard groups and cohomology}
\label{sect:coho}

  Let $X$ be a cubic threefold with $\sA\sD\sE$ singularities, and $\tilde X\to X$ an $\Aut(X)$-equivariant resolution of singularities; it can be achieved via a sequence of blowups, where at each step we blow up the necessarily $\Aut(X)$-invariant singular locus consisting of all the singular points.

  Here we consider the induced $G$-actions on the Picard group $\Pic(\tilde{X})$ and $\Cl(X)$, for $G\subseteq \Aut(X)$. In particular, if the $G$-action on $X$ is linearizable, then the $G$-module $\Pic(\tilde X)$ is a {\em stably permutation module}, see \cite[Section 2]{CTZ}.
  If the cohomology groups
  \[\rH^1(G,\Pic(\tilde{X})), \text{ or } \rH^1(G, \Pic(\tilde{X})^\vee)\] are nonvanishing, then $\Pic(\tilde{X})$ fails to be a stably permutation module. We call this the $(\bf{H1})$-obstruction to linearizability. 
  This is also an obstruction to {\em stable linearizability}, i.e., linearizability of $X\times \bP^n$, with trivial action on the second factor.  
  We refer the reader to \cite[Section 2]{CTZ, CTZcub} for applications. 

  The following proposition shows that the only possible combinations of nonnodal singularities with $(\bf{H1})$-obstructions are 
  $$
 2\sA_5,\quad 2\sD_4+2\sA_1 \quad \text{and}\quad 3\sD_4.
$$
  

\begin{prop}\label{prop:h1}
    Let $X$ be a cubic threefold with isolated singularities, and $\tilde X\to X$ an $\Aut(X)$-equivariant resolution of singularities. Then
    \begin{itemize}
        \item 
  $\Pic(\tilde X)$ is a permutation module for  
    $\Aut(X)$ if $X$ is not of one of the following configurations of singularities
    $$
    6\sA_1 \text{ with defect $0$ }, \quad 8\sA_1,\quad 9\sA_1,\quad 10\sA_1,
    $$
    $$
    2\sA_5, \quad 2\sD_4+2\sA_1\quad\text{ and }\quad 3\sD_4.
    $$
    \item For each of the cubic threefolds $X$ with singularities in the list above,
    if the $\Aut(X)$-action does not fix any singular point then it has an $(\bf{H1})$-obstruction.   
      \end{itemize}
\end{prop}
\begin{proof}
    If $\Aut(X)$ fixes a singular point, then the $\Aut(X)$-action on $X$ is linearizable and $\Pic(\tilde X)$ is an $\Aut(X)$-permutation module.  So it suffices to consider the singularity types in the diagram in the introduction. The cases of nodal ones are treated in \cite{CTZ}. Here we treat $X$ with a nonnodal singular point via a case-by-case study. Let $X$ be a cubic threefold with singularities not in the list of the first assertion and denote the defect of $X$ by $d$. Using the analysis of generators of $\Cl(X)$ in Section~\ref{sec: proj}, we find
    \begin{itemize}
        \item When $d=0$, $\Pic(\tilde X)$ is freely generated by the classes of the hyperplane section and the exceptional divisors, permuted by the $\Aut(X)$-action;
        \item When $d=1$ and the singularity type is $3\sA_3$, $\Pic(\tilde X)$ is freely generated by the class of the hyperplane section, one class of the cubic scrolls in $X$ and the classes of the exceptional divisors, permuted by the $\Aut(X)$-action;
        \item When $d=1$ and the singularity type is not $3\sA_3$, $\Pic(\tilde X)$ is freely generated by the classes of the hyperplane section, the unique plane in $X$ and the exceptional divisors, permuted by the $\Aut(X)$-action;
        \item When $d=2$ and the singularity type is $2\sA_3+3\sA_1$, $\Pic(\tilde X)$ is freely generated by the classes of the hyperplane section, two planes in $X$ and the exceptional divisors, permuted by the $\Aut(X)$-action;
         \item When $d=2$ and the singularity type is $2\sA_3+2\sA_1$, $3\sA_3$ or $2\sD_4$, $\Pic(\tilde X)$ is freely generated by three classes of planes in $X$ and the exceptional divisors, permuted by the $\Aut(X)$-action;
        \item When $d=3$ and the singularity type is $2\sA_3+4\sA_1$, the $\Aut(X)$-action on $X$ is linearizable, see Proposition~\ref{prop:6};
        \item When $d=4$ and the singularity type is $2\sD_4+3\sA_1$, then $X$ is $\Aut(X)$-equivariantly birational to a smooth quadric, see Section~\ref{sect:five}. It follows that $\Pic(\tilde X)$ is a permutation module.
    \end{itemize}
    The proof of the second assertion relies on a detailed analysis on $\Aut(X)$ and the geometry of $X$, see Propositions~\ref{prop:coho}, ~\ref{prop:3d4H1} and ~\ref{prop:2D4+2A1H1}.
\end{proof}

The following lemma simplifies computations in subsequent sections. 
\begin{lemm}\label{lemm:h1seq}
Let $G\subseteq \Aut(X)$ be a finite subgroup. Let $\tilde X\to X$ be a $G$-equivariant 
resolution of singularities and $E_i$ the corresponding exceptional divisors. Then:
\begin{itemize}
    \item If $\rH^1(G,\Cl(X))=0$, then  $\rH^1(G,\Pic(\tilde X))=0$.
    \item  If $\rH^2(G,\oplus_i\bZ\cdot E_i)=0$, then  $\rH^1(G,\Pic(\tilde X))=\rH^1(G,\Cl(X))$, where $\oplus_i\bZ\cdot E_i$ is the free $\bZ$-module generated by $E_i$.
\end{itemize}
\begin{proof}
    We have a short exact sequence 
    \begin{align}
0\to\oplus_i\bZ\cdot E_i\to \mathrm{Cl}(\tilde X)\simeq\Pic(\tilde X)\to \mathrm{ Cl}(X)\to 0,
\end{align}
giving rise to the long exact sequence 
\begin{multline}\label{eq:longexact}
  \ldots\to\\
  \rH^1(G,\oplus_i\bZ\cdot E_i)\to \rH^1(G,\Pic(\tilde X))\to \rH^1(G,\mathrm{ Cl}(X))\to \rH^2(G,\oplus_i\bZ\cdot E_i)\\
  \to\ldots
\end{multline}
Moreover, $\oplus_i\bZ \cdot E_i$ is naturally a $G$-permutation module, induced by the permutation action on the singular points and the exceptional divisors over those points.  
Therefore, $ \rH^1(G,\oplus_i\bZ\cdot E_i)=0$ and the assertions follow from \eqref{eq:longexact}.
\end{proof}
\end{lemm}

\section{Two singular points}
\label{sect:two}
Assume that the cubic threefold $X\subset \bP^4$ is singular at 
$$
p_1=[1:0:0:0:0],\quad p_2=[0:1:0:0:0].
$$
We are interested in the following combinations of singularity types
$$
2\sA_n,  n=2,3,4,5,\quad 2\sD_4.
$$
Up to a change of coordinates, $X$ is given by
\begin{equation}\label{eqn: 2sing}
x_1x_2x_3+x_1q_1+x_2q_2+f_3=0,
\end{equation}
for some quadratic forms $q_1,q_2\in k[x_4,x_5]$ and a cubic $f_3\in k[x_3,x_4,x_5].$ As in \cite{CTZcub}, we see that $X$ is $\Aut(X)$-birational to the hypersurface $V_4$
$$
z_1z_2=q_1q_2-x_3f_3\subset\bP(2,2,1,1,1),
$$
where $z_1=x_1x_3+q_2$ and $z_2=x_2x_3+q_1$. 
This $V_4$ has 2 singular points of type $\frac{1}{2}(1,1,1)$ at $[1:0:0:0:0]$ and $[0:1:0:0:0]$. 
The blowup $\tilde{V}_4$ of these points yields an $\mathrm{Aut}(X)$-equivariant commutative diagram:
$$
\xymatrix{
&\tilde{V}_4\ar@{->}[dl]\ar@{->}[dr]&\\
V_4\ar@{-->}[rr] && \mathbb{P}^2}
$$
where $V_4\dasharrow \mathbb{P}^2$ is the map induced by the projection to the last three coordinates of $\bP(2,2,1,1,1)$,
and $\tilde{V}_4\to \mathbb{P}^2$ is a conic bundle.
The discriminant curve of the conic bundle is a plane quartic curve 
$$
D=\{q_1q_2-x_3f_3=0\}\subset\bP^2_{x_3,x_4,x_5}.
$$

\subsection*{Singularity type $2 \sA_2$}

Up to isomorphism, $X$ is given by \eqref{eqn: 2sing} where:
\begin{align*}
    q_1=x_4^2, \quad q_2=x_4^2 \text{ or } x_5^2, \, \text{ and } \, f_3\text{ a generic cubic form}.
\end{align*} The discriminant curve $D\subset \bP^2$ of the conic bundle is smooth in either case, and we obtain a natural homomorphism 
$$\gamma:\Aut(X)\rightarrow \Aut(D).$$

\begin{prop}
    Let $X$ be a cubic threefold 
    with $2\sA_2$-singularities. 
    Let $G\subset \Aut(X)$  be a subgroup not fixing any singular point of $X$. Then the $G$-action on $X$ is not linearizable.
\end{prop}

\begin{proof}
The proof is essentially the same as the proof of \cite[Theorem 3.3]{CTZcub}, where the claim was proved for a cubic threefold with two nodes. Namely, the group $G$ contains an element $\iota$ switching the singular points of $X$ such that its actions on $\IJ(\tilde{X})$ and $\IJ(D)$ differ by multiplication by $-1$, which implies that the $G$-action on $X$ is not linearizable. We refer to \cite[Theorem 3.3]{CTZcub} for the details.
\end{proof}

\begin{rema}
    The analysis of the induced actions on intermediate Jacobians does not help to settle the linearizability problem when the singularities are worse than those considered above; in particular, when $\IJ(\tilde{X})\cong \JJ(C)$, for a curve $C$ which is either reducible with rational components, or has $g(C)\leq 2$.
\end{rema}

\subsection*{Singularity type $2 \sA_3$ with no plane} 
Up to isomorphism, $X$ is given by
$$
x_1x_2x_3+x_1x_4^2+x_2q_2+f_3=0,
$$
where
$$
f_3\!=\!t_1x_3^3 +  x_3^2(t_2x_4 + t_3x_5)+x_3(t_4x_4^2+t_5x_5^2+t_6x_4x_5)+t_7x_4^2x_5+t_8x_4x_5^2+t_9x_4^3.
$$
Since $X$ contains no planes and has singular points of type $\sA_3$, we have 
$$
q_2=x_5^2\quad \text{ and } \quad t_9=0.
$$ 
The change of variables
\begin{align}\label{eq:killt7t8}
    x_1\mapsto x_1-\frac{t_7^2x_3}{4}-t_7x_5, \quad x_2\mapsto x_2-\frac{t_8^2x_3}{4}-t_8x_5\\ x_3\mapsto x_3,\quad x_4\mapsto x_4+\frac{t_8x_3}{2}, \quad x_5\mapsto x_5+\frac{t_7x_3}{2}\nonumber
\end{align}
eliminates the terms $x_4^2x_5$, $x_4x_5^2$, and we may assume that $t_7=t_8=0$.

\begin{prop}
\label{prop:2a3-noplane}
    Let $X$ be a cubic threefold with $2\sA_3$-singularities and $d(X)=0$, i.e.,  not containing a plane. Assume that  $\Aut(X)$ does not fix any singular point of $X$. Then, 
    up to isomorphism, $X$ is given by 
      \begin{multline}\label{eq:2A3form1}
    x_1x_2x_3  + x_1x_4^2+ x_2x_5^2 + t_1x_3^3 +  x_3^2(t_2x_4 + t_2x_5)+\\+x_3(t_4x_4^2+t_4x_5^2+t_6x_4x_5)=0,
    \end{multline} 
    where $t_1,t_2,t_4,t_6\in k$ and $(\Aut(X),X)$ is one of the following: 
\begin{itemize}
  \item 
  $\Aut(X)=\langle\sigma_{(12)(45)},\eta_1,\eta_2\rangle\simeq\fD_4$, 
  for general $t_1, t_4\in k$ and $t_2=t_6=0$, generated by
    \begin{align*}
    \sigma_{(12)(45)}&:    (\mathbf{x})\mapsto(x_2,x_1,x_3,x_5,x_4),\\
      \eta_1&:(\mathbf{x})\mapsto(x_1,x_2,x_3,-x_4,-x_5),\\
          \eta_2&:(\mathbf{x})\mapsto(x_1,x_2,x_3,x_4,-x_5).
    \end{align*}
     \item $\Aut(X)=\langle\sigma_{(12)(45)},\eta_1\rangle\simeq C_2^2$, 
 for general $t_1,t_4,t_6\in k$ and $t_2=0$.
\item $\Aut(X)=\langle\sigma_{(12)(45)}\rangle\simeq C_2$, for general $t_1,t_2,t_4,t_6\in k$. 
    \end{itemize}
   \end{prop}
\begin{proof}
    We follow the algorithm from Section~\ref{sect:aut}. 
   Let $f$ be the defining equation of $X$, i.e.,
$$
   f=x_1x_2x_3  + x_1x_4^2+ x_2x_5^2 + t_1x_3^3 +  x_3^2(t_2x_4 + t_3x_5)+x_3(t_4x_4^2+t_5x_5^2+t_6x_4x_5),
$$
    and $\iota\in \Aut(X)$ an element switching the two singular points. Based on the form of $f$, one observes that $\iota$ takes the form 
    $$
    \iota=\begin{pmatrix}
        0&s_1&0&0&0\\
        s_2&0&0&0&0\\
        a_1&a_4&1&a_7&a_{10}\\
        a_2&a_5&0&a_8&a_{11}\\
        a_3&a_6&0&a_9&a_{12}
    \end{pmatrix},\quad s_1,s_2\in k^\times, \,\,a_1,\ldots, a_{12}\in k,
    $$
   and $\iota^*(f)=s_1s_2f$. This leads to a system of 24 equations in 20 variables, starting with:
 {\small
    \begin{align*}  
    &s_1a_{12}^2=0,\quad s_2a_8^2=0,\quad s_2a_5 + 2s_2a_7a_8=0,\quad  a_3a_9^2 + a_6a_{12}^2=0,\\ 
    &2a_2a_8a_9 + a_3a_8^2 + 2a_5a_{11}a_{12} + a_6a_{11}^2=0,\\
    &a_2a_9^2 + 2a_3a_8a_9 + a_5a_{12}^2 + 2a_6a_{11}a_{12}=0,\\  
    &s_1a_2 + 2s_1a_{10}a_{11}=0,\quad
   s_2a_6 + 2s_2a_7a_9=0, \\ 
   &s_1a_1 + s_1a_{10}^2=0,\quad
    s_2a_4 + s_2a_7^2=0,\quad... 
\end{align*}
}
These quickly imply (in order)
$$
a_{12}=a_8=a_5=a_3=a_6=a_2=a_{10}=a_{7}=a_1=a_4=0;
$$
it remains to solve the system of equations given by the vanishing of:
    {\small
    \begin{align*} 
    &s_1s_2t_1 - t_1,\quad
    s_1s_2t_2 - t_3a_{11},\quad
    s_1s_2 - s_1a_{11}^2,\quad
    s_1s_2t_4 - t_5a_{11}^2,\\
    &s_1s_2t_3 - t_2a_9,\quad
    s_1s_2t_6 - t_6a_9a_{11},\quad
    s_1s_2 - s_2a_9^2,\quad
    s_1s_2t_5 - t_4a_9^2.
    \end{align*}
    }We do this using the {\tt magma} function {\tt ProbableRadicalDecomposition}.
   Excluding solutions giving rise to cubics with other singularity types, we found that  $\iota$ exists if and only if
$$
s_1 - a_9^2=
        s_2 - a_{11}^2=
        t_2 - t_3a_{11}=
        t_4 - t_5a_{11}^2=
        a_9a_{11} - 1=0.
        $$
Up to a scaling of $x_1,\ldots,x_5$, we may assume that $t_2=t_3$ and $t_4=t_5$. Under these conditions, we find all possibilities for the subgroup $H$ not fixing singular points. In particular, any element $\eta\in H$ takes the form 
     $$
   \eta=\begin{pmatrix}
        s_1&0&0&0&0\\
        0&s_2&0&0&0\\
        a_1&a_4&1&a_7&a_{10}\\
        a_2&a_5&0&a_8&a_{11}\\
        a_3&a_6&0&a_9&a_{12}
    \end{pmatrix},\quad s_1,s_2\in k^\times, \,\,a_1,\ldots, a_{12}\in k.
    $$
    The equality $\eta^*(f)=s_1s_2f$ gives another system of equations. The same method as above yields:
    \begin{itemize}
        \item $t_2=t_3=0,s_1=s_2=a_{12}=1,a_8=-1$, or
        \item $t_2=t_3=t_6=0,s_1=s_2=a_8=1,a_{12}=-1$, 
    \end{itemize}
and all remaining $a_j$ vanish. 
\end{proof}

The following proposition relies on notation from Proposition~\ref{prop:2a3-noplane}. 

\begin{prop}
\label{prop:2A3noplanespec}
 The $\langle\sigma_{(12)(45)}\rangle$-action from Proposition~\ref{prop:2a3-noplane} on a very general cubic threefold $X$ with $2\sA_3$-singularities and defect $d(X)=0$ is not stably linearizable.
\end{prop}

\begin{proof}
We use specialization, as in \cite[Proposition 2.9]{CTZcub}, applied to a higher-dimensional family. Fixing $t_1\in k^\times$ and $t_6\in k$, we consider the family of cubic threefolds
    $$
    \cX\to \bA^2_{t_2,t_4},
    $$
    where the fiber $X_{t_2,t_4}\subset \bP^4$ is given by 
   \begin{multline}\label{eq:2A3special1}
    x_1x_2x_3  + x_1x_4^2+ x_2x_5^2 + t_1x_3^3 +  x_3^2(t_2x_4 +t_2x_5)+\\+x_3(t_4x_4^2+t_4x_5^2+t_6x_4x_5)=0.
    \end{multline} 
    The $\sigma_{(12)(45)}$-action naturally extends to $\cX$.
    For very general $t_2,t_4\in k$, the fiber $X_{t_2,t_4}$ is a cubic threefold with 
    $2\sA_3$-singularities.
The special fiber $X_{0,0}$, at $t_2=t_4=0$, has 
$2\sA_5$-singularities. The $\sigma_{(12)(45)}$-action on $X_{0,0}$ is not stably linearizable, by Proposition~\ref{prop:coho}.

To apply specialization, we resolve, equivariantly, the singularities of the generic fiber of the family $\cX$ via blowing up the $2\sA_3$-points twice. 
This brings us into the situation of a smooth generic fiber and $BG$-rational singularities in the special fiber: the special fiber has $2\sA_1$-singularities in the same $\langle\sigma_{(12)(45)}\rangle$-orbit. The argument works for any fixed $t_1\in k^\times$ and $t_6\in k$, thus, applying specialization, we conclude
the $\langle\sigma_{(12)(45)}\rangle$-action on a very general cubic given by \eqref{eq:2A3form1} is not stably linearizable.
\end{proof}

\begin{coro}\label{coro:2A3noplanelin}
    A $G$-action on a very general cubic threefold in each of the three cases in Proposition~\ref{prop:2a3-noplane} is not stably linearizable if and only if it does not fix two singular points, except possibly one case: $\Aut(X)=\fD_4$ and $G=\langle\sigma_{(12)(45)} \eta_2\rangle\simeq C_4$.
\end{coro}
\begin{proof}
    Any action switching two singular points, except the one specified in the assertion, specializes to an action on a cubic with $2\sA_5$-singularities such that there are {\bf (H1)} obstructions, as in the proof of Proposition~\ref{prop:2A3noplanespec}.
\end{proof}
\begin{rema}
    The exceptional case described in Corollary~\ref{coro:2A3noplanelin} also specializes to a cubic with $2\sA_5$-singularities, but to the group satisfying {\bf (H1)}, and is linearizable, see Proposition~\eqref{prop:tau-iota}.
\end{rema}

\subsection*{Singularity type $2\sA_3$ containing a plane}
Similar to the case with no plane, $X$ is given by
$$
x_1x_2x_3+x_1x_4^2+x_2q_2+f_3=0,
$$
where
$$
     f_3=t_1x_3^3 +  x_3^2(t_2x_4 + t_3x_5)+x_3(t_4x_4^2+t_5x_5^2+t_6x_4x_5)+t_7x_4^2x_5+t_8x_4x_5^2+t_9x_4^3.
$$
Since $X$ contains a plane, we see that
$$
q_2=x_4^2,\quad t_8\ne0.
$$  
Up to a change of variables,
one may assume that  $t_6=t_7=t_9=0$.

\begin{prop}
\label{prop:2a3-plane}
 Let $X$ be a cubic threefold with $2\sA_3$-singularities and $d(X)=1,$ i.e., containing a plane. Then, up to isomorphism, $X$ is given by
 \begin{multline}\label{eq:2A3form2}
    x_1x_2x_3  + (x_1+x_2)x_4^2 + t_1x_3^3 +  x_3^2(t_2x_4 + t_3x_5)+\\+x_3(t_4x_4^2+t_5x_5^2)+t_8x_4x_5^2=0.
    \end{multline} 
Assume that  $\Aut(X)$ does not fix any singular point. Then, up to isomorphism, $(\Aut(X),X)$ is one of the following: 
\begin{itemize} 
\item $\Aut(X)=\langle\sigma_{(12)},\eta_3\rangle\simeq C_2\times C_8$, for general $t_1,t_8\in k^\times$, and
$ t_2=t_3=t_4=t_5=0,
$
 generated by
\begin{align*}
   \sigma_{(12)}&:(\mathbf{x})\mapsto(x_2,x_1,x_3,x_4,x_5),\\
  \eta_3&:(\mathbf{x})\mapsto(-x_1,-x_2,x_3,-\zeta_8^2x_4,\zeta_8x_5).
\end{align*}

\item $\Aut(X)=\langle\sigma_{(12)},\eta_3^2\rangle\simeq C_2\times C_4$, for general  $ t_4\in k$, $t_1,t_8\in k^\times$, and
$
    t_2=t_3=t_5=0.
$
   
\item $\Aut(X)=\langle\sigma_{(12)},\eta_3^4\rangle\simeq C_2\times C_2$, for general  $t_1,t_2,t_4,t_5\in k$, $t_8\in k^\times$ and
$t_3=0.$

\item $\Aut(X)=\langle\sigma_{(12)}\rangle\simeq C_2$, for general  $ t_1,t_2,t_4,t_5\in k$, $t_3, t_8\in k^\times$.

     \end{itemize}
\end{prop}
\begin{proof}
We apply the algorithm of Section~\ref{sect:aut}, as in Proposition~\ref{prop:2a3-noplane}.
\end{proof}




\begin{prop}
\label{prop:a3planex22}
Let $X$ be a cubic threefold with $2\sA_3$-singularities and $d(X)=1$. Let $G\subseteq \Aut(X)$ be a finite subgroup. Then the $G$-action on $X$ is not linearizable if and only if no singular points are fixed by $G$ and $X$ does not contain a $G$-invariant line disjoint from the unique plane $\Pi\subset X$.
\end{prop}

\begin{proof}
Unprojection from the plane $\Pi$ produces, equivariantly, a smooth intersection of two quadrics $X_{2,2}\subset \bP^5$. By \cite{HT-quad}, it is linearizable if and only if $X_{2,2}$ contains $G$-invariant lines. This is equivalent to $G$ fixing a singular point or leaving invariant a line disjoint from $\Pi$, see \cite[Proposition 5.6]{CTZcub}.
\end{proof}

\begin{coro}
\label{cor:linearizable-a3}
    Let $X$ be a cubic threefold with $2\sA_3$-singularities and $d(X)=1.$ Then the $G$-action on $X$ is linearizable if and only if $G$ fixes a singular point or $G=\langle \sigma_{(12)}\rangle.$
\end{coro}
\begin{proof}
For $G=\langle\sigma_{(12)}\rangle$, the $G$-fixed locus on $X$ is a smooth cubic surface. Its image under the unprojection to $X_{2,2}$ is a smooth del Pezzo surface of degree 4, with 16 lines. Then $G$ is linearizable. All the other possible subgroups $G$ in Proposition~\ref{prop:2a3-plane} not fixing any singular points contain an element of the form $\sigma_{(12)}\eta_3^r.$ One can check that for all $r=1,\ldots, 7$, $\sigma_{(12)}\eta_3^r$ does not leave invariant any line in $X$ disjoint from $\Pi$. Therefore, the corresponding $G$-action is not linearizable by Proposition~\ref{prop:a3planex22}.
\end{proof}


\subsection*{Singularity type $2 \sA_4$} Up to isomorphism, $X$ is given by
\begin{multline*}
    q_1=x_4^2, \quad q_2=x_5^2,\quad f_3=t_1x_3^3+x_3^2(t_2x_4+t_3x_5)+\\+x_3(-\frac{t_7^2}{4}x_4^2-\frac{t_8^2}{4}x_5^2+t_6x_4x_5)
    +t_7x_4^2x_5+t_8x_5^2x_4
\end{multline*}
for general parameters $t_1,t_2,t_3,t_6,t_7,t_8\in k.$ As above, we may assume that $t_7=t_8=0$ and $t_2=t_3$, up to a change of variables.

\begin{prop}\label{prop:2A4auto}
    Let $X$ be a cubic threefold with $2\sA_4$-singularities. Assume  that $\Aut(X)$ does not fix a singular point. Then, up to isomorphism, $X$ is given by 
      \begin{multline}\label{eq:2A4form1}
    x_1x_2x_3  + x_1x_4^2+ x_2x_5^2 + t_1x_3^3 +  x_3^2(t_2x_4 + t_2x_5)+t_6x_3x_4x_5=0,
    \end{multline} 
    with general parameters $t_1,t_2,t_6\in k$ and $(\Aut(X),X)$ is one of the following: 
    \begin{itemize}
 \item  
    $\Aut(X)=\langle\sigma_{(12)(45)},\eta_4\rangle\simeq C_6$,
 for general $t_2\in k^\times$ and
  $
  t_1=t_6=0,
  $ generated by 
\begin{align*}
   \sigma_{(12)(45)}&:    (\mathbf{x})\mapsto(x_2,x_1,x_3,x_5,x_4),\\
  \eta_4&:(\mathbf{x})\mapsto(\zeta_3^2x_1,\zeta_3^2x_2,x_3,\zeta_3x_4,\zeta_3x_5).
\end{align*}
    
   \item $\Aut(X)=\langle\sigma_{(12)(45)}\rangle\simeq C_2$
    for general parameters $t_1,t_2,t_6\in k$.
    \end{itemize}
    \end{prop}
\begin{proof}
    Similar to the proof of Proposition~\ref{prop:2a3-noplane}.
\end{proof}
    
\begin{prop}\label{prop:2A4spec2A5}
  Let $X$ be a very general cubic with $2\sA_4$-singularities such that $\Aut(X)$ switches two singular points. Then the $\langle\sigma_{(12)(45)}\rangle$-action on $X$ is not stably linearizable.
\end{prop}
\begin{proof}
     By the classification in Proposition~\ref{prop:2A4auto},  it suffices to show the $\langle\sigma_{(12)(45)}\rangle$-action on a very general $2\sA_4$ cubic threefold $X$ is not stably linearizable. We use specialization, as in \ref{prop:2A3noplanespec}. 

    Fix $t_6\in k$ and $t_1\in k^\times$, and consider the family  of cubic threefolds
    $$
    \pi: \cX\to \bA^1_{t_2}
    $$
    whose generic fiber $X_{t_2}:=\cX_{t_2}$ is the cubic threefold given by 
    $$
x_1x_2x_3+x_1x_4^2+x_2x_5^2+t_1x_3^3+x_3^2(t_2x_4+t_2x_5)+t_6x_3x_4x_5=0.
    $$
    The $\sigma_{(12)(45)}$-action extends to $\cX$.
     For very general $t_2\in k$, the fiber $X_{t_2}$ is a cubic threefold with $2\sA_4$-singularities.
    The special fiber $X_0$ at $t_2=0$ has 2$\sA_5$-singularities. Moreover, by Proposition~\ref{prop:coho}, the $\langle\sigma_{(12)(45)}\rangle$-action on $X_0$ is not stably linearizable.     As in Proposition~\ref{prop:2A3noplanespec}, applying specialization to a resolution of singularities of the generic fiber of the family $\cX$ completes the proof.

\end{proof}


\subsection*{Singularity type $2 \sA_5$} According to \cite{All}, 
see also \cite[Theorem 3.2(iii)]{ACT}, any cubic threefold $X_b$ with $2\sA_5$-singularities is given by 
\begin{equation}
    \label{eqn:2a5}
X_b=\{x_1x_2x_3+x_1x_4^2+x_2x_5^2+x_3^3+bx_3x_4x_5=0\},\quad b^2\ne-4.
\end{equation}
One has 
$$
\Aut(X_b)=\begin{cases}
    \langle \tau_a,\sigma_{(12)(45)}\rangle\simeq \bG_m(k) \rtimes C_2, & b^2\ne 0, -4,\\
     \langle \tau_a,\sigma_{(12)(45)},\eta_2\rangle\simeq(C_2\times \bG_m(k))\rtimes C_2, &b=0,
\end{cases}
$$
where 
\begin{align}\label{eq:2A5iota}
    \tau_a: ({\bf x})&\mapsto(a^2x_1,a^{-2}x_2,x_3,a^{-1}x_4,ax_5),\quad a\in k^\times ,\nonumber\\
   \sigma_{(12)(45)}:({\bf x})&\mapsto (x_2,x_1,x_3,x_5,x_4),\\
        \eta_2:({\bf x})&\mapsto (x_1,x_2,x_3,x_4,-x_5).\nonumber
\end{align}

\subsection*{Cohomology}
By results in Section~\ref{sect:defect}, the defect of $X_b$ with $2\sA_5$-singularities equals $1$, and $\Cl(X)$ is generated by two classes of rational normal cubic scrolls in $X$. Projecting from $q=[1:0:0:0:0]$, we see that the associated $(2,3)$-curve 
$$
R_q=\{x_2x_3+x_4^2= x_2x_5^2+bx_3x_4x_5+x_3^3=0\}\subset\bP^3_{x_2,x_3,x_4,x_5}
$$
is the union of two twisted cubic curves, given by 
\begin{multline*}
   R_1=\{x_2x_3+x_4^2=x_3^2-\frac{-b+\sqrt{b^2+4}}{2}x_4x_5= \\
   =x_3x_4+\frac{-b+\sqrt{b^2+4}}{2}x_2x_5=0\}\subset\bP^3_{x_2,x_3,x_4,x_5}
\end{multline*}
and
\begin{multline*}
   R_2=\{x_2x_3+x_4^2=x_3^2-\frac{-b-\sqrt{b^2+4}}{2}x_4x_5= \\
   =x_3x_4+\frac{-b-\sqrt{b^2+4}}{2}x_2x_5=0\}\subset\bP^3_{x_2,x_3,x_4,x_5}.
\end{multline*}
Let $\widehat R_1$, respectively $\widehat R_2$, be the cones over $R_1$, respectively $R_2$. The classes of $\widehat R_1$ and $\widehat R_2$ in $\Cl(X)$ give another set of generators of $\Cl(X)$, equivalent to the classes of two cubic scrolls.
\begin{prop}\label{prop:coho}
    Let $X$ be a cubic threefold of singularity type $2\sA_5$, and $G=\langle \sigma_{(12)(45)}\rangle$ given by \eqref{eq:2A5iota}.
    Then 
    $$
    \rH^1(G,\Pic(\tilde{X})) =\bZ/2.
    $$
\end{prop}
\begin{proof}
    First, one checks that $\widehat R_1\cup\sigma_{(12)(45)}(\widehat R_1)$ is cut out by the quadric hypersurface section of $X$ given by 
    $$
    \widehat R_1\cup\sigma_{(12)(45)}(\widehat R_1)=\{x_3^2 + \frac{b-\sqrt{b^2+4}}{2}x_4x_5=0\}\cap X.
    $$
    This implies that $\sigma_{(12)(45)}$ switches the two generators of $\Cl(X)$.
    As in \cite[Proposition 7.5]{CTZcub}, we 
   compute
    $$
    \rH^1(G,\Pic(\tilde{X})) =\bZ/2.
    $$
\end{proof}


\subsection*{Linearizability}
When $b=0$, the action of $\eta_2\cdot\sigma_{(12)(45)}$ switches two nodes and has vanishing cohomology. This action is linearizable:

\begin{prop}
\label{prop:tau-iota}
    Let $X$ be the cubic threefold given by \eqref{eqn:2a5} with $b=0$ and $G\simeq \bG_m(k)\rtimes C_2$, generated by 
$\eta_2\sigma_{(12)(45)}$ and $\tau_a, a \in k^\times$. Then the $G$-action on $X$ is linearizable.
\end{prop}

\begin{proof}
Recall that $X=\{x_1x_2x_3+x_1x_4^2+x_2x_5^2+x_3^3=0\}$
and 
\begin{align*}
\eta_2\sigma_{(12)(45)}&: \mathbf{(x)}\mapsto (x_2,x_1,x_3,x_5,-x_4),\\
\tau_a&: \mathbf{(x)}\mapsto ( a^2x_1, a^{-2}x_2, x_3, a^{-1}x_4, ax_5),\quad a\in k^\times.
\end{align*}
in particular, it leaves the affine chart $\{x_3\neq 0\}$ invariant. Thus we can assume that $x_3=1$, and consider the $G$-equivariant change of coordinates 
    $$
    y_1=x_1+x_5^2, \quad y_2=x_2+x_4^2,
    $$
yielding the equation 
    $$
    y_1y_2+(1-x_4x_5)(1+x_4x_5)=0.
    $$ 
    Let
    $$
    z_1=\frac{y_1}{(1+x_4x_5)},\quad z_2=\frac{y_2}{(1-x_4x_5)};
    $$ 
    this $G$-equivariant birational change of coordinates gives a $G$-birational map $X\dasharrow Y$, where 
    $$Y=\{z_1z_2+t^2=0\}\subset \bP^4_{z_1, z_2, x_4, x_5,t}.
    $$ 
    Thus $Y$ is a cone over a smooth conic, with $G$-action generated by
    \begin{align*}
    \eta_2\sigma_{(12)(45)}:& (z_1,z_2, x_4, x_5,t)\mapsto (z_2, z_1, x_5, -x_4,t),\\
    \tau_a:& (z_1,z_2, x_4, x_5,t)\mapsto (a^2z_1, a^{-2}z_2, x_4, x_5,t).
    \end{align*}
     Projecting from the $G$-fixed point $[0:0:1:\zeta_4:0]$,
    we obtain linearization.
\end{proof}
Combining Proposition~\ref{prop:coho} and ~\ref{prop:tau-iota}, we settle the linearizability problem of cubic threefolds with $2\sA_5$-singularities: 

\begin{coro}\label{coro:2a5sum}
    Let $X$ be a cubic threefold with $2\sA_5$-singularities given by \eqref{eqn:2a5} and $G\subseteq \Aut(X)$. Then the $G$-action on $X$ is not (stably) linearizable if and only if $G$ contains an element conjugate to $\sigma_{(12)(45)}$ given by \eqref{eq:2A5iota}.
\end{coro}
\begin{proof}
    By Proposition~\ref{prop:coho}, the $G$-action is not stably linearizable if $\sigma_{(12)(45)}\in G$. When $G$ switches two singular points but does not contain any element conjugate to $\sigma_{(12)(45)}$, we are in the situation where $b=0$ in \eqref{eqn:2a5} and $G$ is a subgroup of the group generated by 
$\eta_2\sigma_{(12)(45)}$ and $\tau_a, a \in k^\times$. Such $G$-actions are linearizable, by Proposition~\ref{prop:tau-iota}. 
\end{proof}
\subsection*{Singularity type $2\sD_4$}

\begin{prop}\label{prop:2D4class}
    Let $X$ be a cubic threefold with $2\sD_4$-singularities. Up to isomorphism, $X$ is given by 
    \begin{align}\label{eq:2D4form}
         x_1x_2x_3+f_3(x_3,x_4,x_5)=0,
    \end{align}
    where $f_3$ is a generic cubic form in $x_3,x_4,x_5$, i.e.,
    \begin{align}\label{eq:2d4f3eq}
        f_3&=t_1x_3^3 +  x_3^2h_1+x_3h_2+h_3,\\
        h_1&=t_2x_4 + t_3x_5,\nonumber\\
        h_2&=t_4x_4^2+t_5x_5^2+t_6x_4x_5,\nonumber\\
h_3&=t_7x_4^2x_5+t_8x_4x_5^2+t_9x_4^3+t_{10}x_5^3\nonumber
    \end{align}
   for general $t_1,\ldots,t_{10}\in k$, with $\Aut(X)$ one of the following:
    \begin{itemize}
        \item[Case (1):] $\Aut(X)=\langle \sigma_{(12)},\tau_a, \eta_1,\eta_3,\sigma_{(45)}\rangle \simeq (\bG_m(k)\rtimes C_2)\times\fS_3\times C_3 $, $X$ is given by 
        $$
x_1x_2x_3+x_3^3+x_4^3+x_5^3=0.
        $$
     \item[Case (2):] $\Aut(X)=\langle \sigma_{(12)},\tau_a, \eta_1,\sigma_{(45)}\rangle\simeq(\bG_m(k)\rtimes C_2)\times\fS_3$, $X$ is given by 
        $$
x_1x_2x_3+x_3^3+t_6x_3x_4x_5+x_4^3+x_5^3=0,
        $$
        for general $ t_6\in k$.
          \item[Case (3):] $\Aut(X)=\langle \sigma_{(12)},\tau_a, \eta_3,\sigma_{(45)}\rangle\simeq(\bG_m(k)\rtimes C_2)\times C_6$, $X$ is given by 
        $$
x_1x_2x_3+x_3^3+(x_4+x_5)(r_4x_4+r_5x_5)(r_5x_4+r_4x_5)=0,
        $$
        for general $ r_4, r_5\in k$.
         \item[Case (4):] $\Aut(X)=\langle \sigma_{(12)},\tau_a, \sigma_{(45)}\rangle\simeq(\bG_m(k)\rtimes C_2)\times C_2$, $X$ is given by 
        \begin{multline*}
x_1x_2x_3+t_1x_3^3+r_1x_3^2(x_4+x_5)+x_3(r_2x_4+r_3x_5)(r_3x_4+r_2x_5)+\\
+(x_4+x_5)(r_4x_4+r_5x_5)(r_5x_4+r_4x_5)=0,
        \end{multline*}
        for general  $t_1,r_1,\ldots,r_5\in k$.
         \item[Case (5):] $\Aut(X)=\langle \sigma_{(12)},\tau_a, \eta_2\rangle\simeq(\bG_m(k)\rtimes C_2)\rtimes C_2$, $X$ is given by 
$$
x_1x_2x_3+x_3^2(t_2x_4+t_3x_5)
+t_7x_4^2x_5
+t_8x_4x_5^2+t_9x_4^3+t_{10}x_5^3=0,
        $$
        for general $t_2,t_3, t_7, t_8,t_9,t_{10}\in k$.
          \item[Case (6):] $\Aut(X)=\langle \sigma_{(12)},\tau_a \rangle\simeq \bG_m(k)\rtimes C_2$, $X$ is given by 
$$
\text{vanishing of \eqref{eq:2D4form} where } f_3 \text{ is a generic cubic form }
     $$
      such that $h_2\not\equiv 0$,
    \end{itemize}
    
    where
        \begin{align*}
        \sigma_{(12)}: (\mathbf x)&\mapsto (x_2, x_1, x_3, x_4, x_5),\\
         \tau_a: (\mathbf x)&\mapsto (a x_1, a^{-1}x_2, x_3,x_4,x_5),\quad a\in k^\times,\\
         \eta_1: (\mathbf x)&\mapsto (x_1, x_2, x_3,\zeta_3 x_4,\zeta_3^2x_5),\\
          \eta_2: (\mathbf x)&\mapsto (-x_1, x_2, x_3,-x_4,-x_5),\\
           \eta_3: (\mathbf x)&\mapsto (x_1, x_2, x_3,\zeta_3x_4,\zeta_3x_5),\\
           \sigma_{(45)}: (\mathbf x)&\mapsto (x_1, x_2, x_3,x_5,x_4).
    \end{align*}
\end{prop}
\begin{proof}
    For any such cubic $X$, $\Aut(X)$ contains a subgroup isomorphic to the infinite dihedral group
    generated by 
    \begin{align}\label{eq:2d4toricaction}
        \sigma_{(12)}: (\mathbf x)&\mapsto (x_2, x_1, x_3, x_4, x_5),\nonumber\\
         \tau_a: (\mathbf x)&\mapsto (a x_1, a^{-1}x_2, x_3,x_4,x_5),\quad a\in k^\times.
    \end{align}
    To find possibilities of $\Aut(X)$, it suffices to find $H\subset\Aut(X)$, the subgroup fixing both singular points. Based on the form of \eqref{eq:2D4form}, one sees that any element in $H$ takes the form 
    $$
    \begin{pmatrix}
         s_1&0&0&0&0\\
       0&s_2&0&0&0\\
       0&0&1&b_1&b_4\\
       0&0&0&b_2&b_5\\
       0&0&0&b_3&b_6,
    \end{pmatrix},\quad 
    s_1,s_2, b_1,\ldots, b_6\in k,
    $$
    Then up to a change of variables only in coordinates $x_4$ and $x_5$,
    we may assume $b_1=b_4=0$ without changing the form of \eqref{eq:2D4form}.
    Namely, $H$ preserves \eqref{eq:2D4form} and acts on the ambient $\bP^4$ via 
    $
    \bP(I_1\oplus I_2\oplus I_3\oplus V), 
    $
    where $I_1, I_2$ and $I_3$ are 1-dimensional representations of $H$, acting respectively on coordinates $x_1, x_2$ and $x_3$, and $V$ is a $2$-dimensional representation of $H$ acting on $x_4,x_5$. In the plane $\bP^2_{x_3,x_4,x_5}$, the group $H$ leaves both the line $l=\{x_3=0\}$ and the cubic curve $C=\{f_3=0\}$ invariant. Since $X$ is a cubic with $2\sD_4$-singularities, by Proposition \ref{prop:defectd4pt}, $X$ contains three distinct planes, corresponding to the points defined by $l\cap C$. This implies $l\cap C$ defines three distinct points, in the same $H$-orbit. Consider the exact sequence 
    $$
    0\to H'\to H\to \bar H\to0
    $$
    where $H'$ contains elements in $H$ acting via scalars in $V$, and $\bar H$ acts faithfully on $\bP(V)=\bP^1_{x_4,x_5}$. Since $H$ leaves invariant three points, the possibilities of $\bar H$ are 
    $$
    \bar H=C_1, C_2, C_3, \text{ or } \fS_3,
    $$
    where $\fS_3$ is generated by 
    $$
    \sigma=\begin{pmatrix}
        \zeta_3&0\\
        0&\zeta_3^2
    \end{pmatrix}\quad \text{and}\quad   \iota =\begin{pmatrix}
       0&1\\
        1&0
    \end{pmatrix},
    $$
    and the other possibilities are the corresponding subgroups of $\fS_3$. Moreover, $H$ leaves invariant each of the following subsets of $\bP^1_{x_4,x_5}$, defined by 
    $$
    Q_1=\{h_1=0\},\quad Q_2=\{ h_2=0\},\quad Q_3=\{h_3=0\}.
    $$
    Using this, we classify the possibilities of $H$ and $\bar H$.

    \

      \noindent{\bf{When} $\bar H=C_1$:} In this case $H=H'.$ We find below all possibilities of $H'$. By definition, any element in $\eta\in H'$ takes the form 
      $$
      \eta: (\mathbf{x})\mapsto (s_1x_1, s_2x_2, x_3, s_3x_4, s_3x_5), \quad s_1,s_2,s_3\in k^\times.
      $$
      The weights of the $\eta$-action  on $h_1, h_2, h_3$ are respectively 
    $s_3, s_3^2, s_3^3$. Since $h_3\not\equiv 0$, there are the following cases:
    \begin{itemize}
        \item When $h_2\not\equiv 0$ : we have $s_3=1$, $\eta$ is the toric action \eqref{eq:2d4toricaction} and $H'\simeq k^\times$.
        \item  When $h_2\equiv 0$,  $h_1\not\equiv 0$ and $t_1=0$: we have $s_3=-1$, $s_1s_2=-1$, and $H'\simeq C_2\times k^\times$.
        \item When $h_2\equiv 0$, $h_1\equiv 0$ and $t_1\ne 0$: we have $s_3^3=1$ and $H'\simeq C_3\times k^\times$.
            \item When $h_2\equiv 0$, $h_1\equiv 0$ and $t_1= 0$: $X$ has $3\sD_4$-singularities.
    \end{itemize}

    \ 
     
    \noindent{\bf{When} $\bar H=\fS_3$:} Since $\fS_3$ has no fixed points in $\bP^1,$ one has $h_1\equiv 0$, $h_2=t_6x_4x_5$ and $h_3=x_4^3+x_5^3$.

    \ 

     \noindent{\bf{When} $\bar H=\langle\sigma\rangle\simeq C_3$:} We know that $Q_1$ is a fixed point of $\bar H$, i.e., $h_3=t_2x_4$ or $t_3x_5.$ Similarly, $Q_2$ can also only contain fixed points of $\bar H$, i.e., $h_3=t_4x_4^2$, $t_5x_5^2$, or $t_6x_4x_5$, and $Q_3$ contains three distinct points in one $\bar H$-orbit, thus, up to scaling, $h_3=x_4^3+x_5^3.$ Matching the weights of the $\sigma$-actions on each of the monomials appearing in $f_3$, one sees that the only choice is $h_1\equiv 0,$ $h_2=t_6x_4x_5$ and $h_3=x_4^3+x_5^3$. Then we go back to the situation above. Thus, $\bar H\not\simeq C_3$, i.e., $\sigma\in \bar H$ implies that $\bar H\simeq \fS_3$.

     \ 
     
 \noindent{\bf{When} $\bar H=\langle\iota\rangle\simeq C_2$:} As above, using symmetries on $\bP^1_{x_4,x_5}$, and matching the weights on the monomials, we find two cases: 
 \begin{itemize}
     \item $h_1=r_1(x_4+x_5)$, for some $r_1\in k$,
     \item $h_2=(r_2x_4+r_3x_5)(r_3x_4+r_2x_5)$, for some $r_2,r_3\in k$,
       \item  $h_3=(x_4+x_5)(r_4x_4+r_5x_5)(r_5x_4+r_4x_5)$, for some $r_4\ne r_5\in k^\times;$
 \end{itemize}
 or 
 \begin{itemize}
     \item $h_1=r_1(x_4-x_5)$, for some $r_1\in k$,
     \item $h_2=r_2(x_4^2-x_5^2)$,  for some $r_2,r_3\in k$,
     \item  $h_3=(x_4-x_5)(r_4x_4+r_5x_5)(r_5x_4+r_4x_5)$, for some $r_4\ne r_5\in k^\times,$
     \item $t_1=0$.
 \end{itemize}
   Combining all possibilities of $\bar H$ and of $H'$, and checking the singularity types, we obtain the assertion. 
\end{proof}


The following applies in Cases(1), (2), (3), and (4) of Proposition~\ref{prop:2D4class}.

\begin{prop}
\label{prop:D4Burn}
Let $X$ be a cubic threefold with $2\sD_4$-singularities admitting the action of 
$H:=\langle \sigma_{(12)},\sigma_{(45)}\rangle$. Then, for any $G\subseteq \Aut(X)$ containing $H$, the $G$-action on $X$ is not linearizable. 
\end{prop}

\begin{proof}
    The $ \sigma_{(12)}$-action fixes a smooth cubic surface $S\subset X$ and the residual $ \sigma_{(45)}$-action fixes a genus $1$ curve on $S$, producing an incompressible symbol, in the terminology of, e.g.,  \cite[Section 3]{TYZ-3}. We conclude as in \cite[Proposition 2.6]{CTZcub}.
\end{proof}

\begin{prop}
\label{prop:D4Burn2}
Let $X$ be a cubic threefold with $2\sD_4$-singularities and $G=\langle \sigma_{(12)},\tau_a\rangle\simeq\fD_{2n}$, $n\ge 2$,  where $a=\zeta_{2n}$ is a primitive $2n$-th root of unity and $ \sigma_{(12)},\tau_a$ are described in Proposition~\ref{prop:2D4class}. Then the $G$-action on $X$ is not linearizable. 
\end{prop}
\begin{proof}
    Recall that $G=\fD_{2n}$ is the dihedral group of order $4n$. Observe that $G$ pointwise fixes a smooth elliptic curve $E=\{x_1=x_2=0\}\subset X$. To apply the Burnside formalism, one has to pass to a standard model, and, in particular, blow up strata with nonabelian generic stabilizers. Thus, one 
   needs to blow up $E$ in $X$, see \cite[Section 7.2]{HKTsmall} for definitions. The exceptional divisor has  generic stabilizer $C_2$. It follows that on a standard form for the action $X\actsfromright G$, we find the symbol 
   \begin{align}\label{eq:incomp2d4genus1}
          (C_2,\fD_{n}\actsfromleft k(S), (1)),
   \end{align}
   where $S=\bP(\cN_{E/X})$, the projectivization of the normal bundle of $E$ in $X$, and in particular, $S$ is a $\bP^1$-bundle over $E$. This symbol is incompressible: the $\fD_{n}$-action on $S$ is not birational to any actions on the blowup of a genus 1 curve with abelian stabilizers. To see this, one can apply the Burnside formalism in dimension 2. Notice that the $\fD_n$-action on $S$ is trivial on the base of the fibration $S\to E$. So we find an {\em incompressible} symbol in the class $[S\actsfromright G]$:
 \begin{align}\label{eq:incomp2d4genus1dim2}
   (C_n,\mathrm{triv}\actsfromleft k(E),(\beta)),
    \end{align}
   for some character $\beta$ of $C_n$. On the other hand, such symbols do not arise from any $\fD_n$-action on the blowup of a genus 1 curve on any standard model -- on a standard model, the curve has abelian stabilizer and receives a nontrivial action from $\fD_n$. It can never produce a divisorial symbol with a trivial residual action as in \eqref{eq:incomp2d4genus1dim2}. Equivariant nonbirationality of $S$ with a blowup of a genus 1 curve with abelian stabilizers also follows from the the functoriality of passage to MRC quotients, see \cite[Theorem IV.5.5]{K-book}. 
Therefore, we conclude that symbols \eqref{eq:incomp2d4genus1} are incompressible. Such symbols do not appear from linear actions on $\bP^3$, which implies that the $G$-action on $X$ is not linearizable. 
\end{proof}

\begin{prop}\label{prop:2d4spec2d4+2a1}
    The $\langle \sigma_{(12)} \sigma_{(45)}\rangle$-action on a very general cubic threefold with $2\sD_4$-singularities described in Case (4) in Proposition~\ref{prop:2D4class} is not stably linearizable.
\end{prop}

\begin{proof}
    Recall that such $X$ with $2\sD_4$-singularities are given by 
         \begin{multline}\label{eq:2d4case4eq}
x_1x_2x_3+t_1x_3^3+r_1x_3^2(x_4+x_5)+x_3(r_2x_4+r_3x_5)(r_3x_4+r_2x_5)+\\
+(x_4+x_5)(r_4x_4+r_5x_5)(r_5x_4+r_4x_5)=0,
        \end{multline}
        for general parameters $t_1,r_1,r_2,r_3,r_4,r_5\in k$, and $\sigma_{(12)} \sigma_{(45)}$ takes the form 
        $$
        \sigma_{(12)} \sigma_{(45)}: \mathbf{(x)}\to(x_2,x_1,x_3,x_5,x_4).
        $$
        Now we view \eqref{eq:2d4case4eq} as a family $\cX\to \bA^6$ of cubic threefolds with $2\sD_4$-singularities parameterized by $t_1,r_1,r_2,r_3,r_4,r_5\in k$. The general fibers above $r_2-r_3=t_1=0$ are cubic threefolds with $2\sD_4+2\sA_1$-singularities. In particular, under the change of variables 
        $$
        y_1=x_1,\quad y_2=x_2,\quad y_3=\frac14(2x_3-\frac{r_4-r_5}{\sqrt{r_1}}x_4+\frac{r_4-r_5}{\sqrt{r_1}}x_5)$$
        $$
        y_4=\frac14(2x_3+\frac{r_4-r_5}{\sqrt{r_1}}x_4-\frac{r_4-r_5}{\sqrt{r_1}}x_5),\quad y_5=x_4+x_5,
        $$
        the fibers above $r_2-r_3=t_1=0$ are given by 
        \begin{align}\label{eq:special2d4+2a1}
y_1y_2y_3+y_1y_2y_4+4r_1y_3y_4y_5+r_2^2y_5^2(y_3+y_4)+\frac14(r_4+r_5)^2y_5^3=0,
        \end{align}
        and $\sigma_{(12)} \sigma_{(45)}$ under the new basis is 
        $$
\sigma_{(12)} \sigma_{(45)}:\mathbf{(y)}\mapsto(y_2,y_1,y_4,y_3,y_5).
        $$
        From Proposition~\ref{prop:2D4+2A1H1}, we see that the $\langle\sigma_{(12)} \sigma_{(45)}\rangle$-action on the cubic given by \eqref{eq:special2d4+2a1} is not stably linearizable. Applying specialization to the resolution of the $2\sD_4$-singularities in the generic fiber of the family $\cX$, we conclude that the $\langle\sigma_{(12)} \sigma_{(45)}\rangle$-action on a very general member in $\cX$ is not stably linearizable.
\end{proof}


\begin{prop}\label{prop:2d4linsigma2}
    Let $X$ be a cubic threefold with $2\sD_4$-singularities. Then the $\langle\sigma_{(12)}\tau_a\rangle$-action on $X$ from Proposition~\ref{prop:2D4class} is linearizable for any $a\in k^\times$.
\end{prop}

\begin{proof}
    The $\sigma_{(12)}\tau_a$-action preserves each of the three planes and pointwise fixes a smooth cubic surface 
    $$
    S=\{x_1-ax_2=0\}\cap X
    $$
for any $a\in k^\times$. Unprojection from one plane birationally transforms $X$ to an intersection of two quadrics $X_{2,2}$ in $\bP^5$, with $2\sA_1$-singularities. The cubic surface $S$ becomes a smooth del Pezzo surface of degree $4$ in $X_{2,2}$, and it contains 16 lines, fixed by the action. Projection from any of the lines yields a linearization of the $\sigma_{(12)}\tau_a$-action on $X$. 
\end{proof}

\section{Three singular points}
\label{sect:three}
Let $X$ have three singular points. We may assume that they are at 
$$
p_1=[1:0:0:0:0],\quad p_2:=[0:1:0:0:0],\quad p_3:=[0:0:1:0:0], 
$$
so that $X$ is given by 
\begin{align}\label{eq:3sin}
    x_1x_2x_3+x_1q_1+x_2q_2+x_3q_3+f_3=0,
\end{align}
where $q_1,q_2,q_3,f_3\in k[x_4,x_5]$. There are three possibilities:
$$
3\sA_2,\quad 3\sA_3,\quad 3\sD_4.
$$
All of these are specializations of the $3\sA_1$ case, studied in \cite[Section 4]{CTZcub}. Here, we use similar arguments.

\subsection*{Singularity Types $3\sA_2$ and $3\sA_3$}  Since $p_1, p_2$ and $p_3$ are $\sA_n$-points with $n=2,3$, the rank of $q_1,q_2,q_3$ is $1$, i.e., $q_i=l_i^2$ for some linear forms $l_i\in k[x_4,x_5]$, $i=1,2,3$. Observe that if the singularity type is $3\sA_2$, then $X$ contains no plane. It follows that $q_1, q_2, q_3$ and $f_3$ do not share a common factor.

\begin{prop}\label{prop:3Anauto}
Let $X$ be a cubic threefold with singularity types $3\sA_2$ or $3\sA_3$. Assume that $\Aut(X)$ does not fix any singular points.

If $X$ has $3\sA_2$-points then either:
\begin{enumerate}
    \item $\Aut(X)=\langle\sigma_{(123)},\sigma_{(23)},\eta_1\rangle\simeq C_3\times\fS_3$, where 
    \begin{align*}
        \sigma_{(123)}&:(\mathbf{x})\mapsto (x_3, x_1,x_2,x_4,x_5),\\
         \sigma_{(23)}&:(\mathbf{x})\mapsto (x_1, x_3,x_2,x_4,x_5),\\
          \eta_1&:(\mathbf{x})\mapsto (x_1, x_2,x_3,x_4,\zeta_3x_5),
    \end{align*}
    and $X$ is given by 
    $$
    x_1x_2x_3+x_4^2(x_1+x_2+x_3) + ax_4^3+x_5^3=0,
      $$
for general  $a\in k^\times$;
    \item $\Aut(X)=\langle\sigma_{(123)},\sigma_{(23)},\eta_2\rangle\simeq C_6\times\fS_3$, where 
    \begin{align*}
          \eta_2&:(\mathbf{x})\mapsto (x_1, x_2,x_3,-x_4,\zeta_3x_5),
    \end{align*}
    and $X$ is given by 
    \begin{align*}
    x_1x_2x_3+x_4^2(x_1+x_2+x_3)+x_5^3=0.
    \end{align*}
     \item $\Aut(X)=\langle\sigma_{(123)},\sigma_{(23)}\rangle\simeq \fS_3$,
    and $X$ is given by
   \begin{align*}
    & x_1x_2x_3+x_4^2(x_1+x_2+x_3)+f_3=0,\\
    \text{with}&\,\,f_3\in k[x_4, x_5] \text{ a general cubic form};
    \end{align*}
      \item $\Aut(X)=\langle\sigma_{(123)}\eta_1\rangle\simeq C_3$, and $X$ is given by 
  \begin{align*}
&x_1x_2x_3+x_1(x_4+x_5)^2+x_2(x_4+\zeta_3x_5)^2+x_3(x_4+\zeta_3^2 x_5)^2+f_3=0,
    \end{align*}
    where $f_3=ax_4^3+bx_5^3$, for $a\ne b\in k$,  or $f_3=cx_4^3$, for $c\in k^\times;$
   
\end{enumerate}

If $X$ has $3\sA_3$-points then:
\begin{enumerate}
    \item $\Aut(X)=\langle\sigma_{(123)},\sigma_{(23)},\eta_3\rangle\simeq C_4\times\fS_3$, where 
    \begin{align*}
          \eta_3&:(\mathbf{x})\mapsto (x_1, x_2,x_3,-x_4,\zeta_4x_5),
    \end{align*}
    and $X$ is given by 
    \begin{align*}
    x_1x_2x_3+x_4^2(x_1+x_2+x_3)+ax_4x_5^2=0,
    \end{align*}
    for general $a\in k^\times$.
      \item $\Aut(X)=\langle\sigma_{(123)}\eta_1, \sigma\rangle\simeq \fS_3$, where 
      $$
      \sigma:(\mathbf{x})\mapsto(x_1,\zeta_3x_3,\zeta_3^2x_2,x_5,x_4),
      $$
    and $X$ is given by 
  \begin{multline*}
  x_1x_2x_3+x_1(x_4+x_5)^2+x_2(x_4+\zeta_3x_5)^2+x_3(x_4+\zeta_3^2 x_5)^2+a(x_4^3+x_5^3)=0,
   \end{multline*}
   for general $a\in k^\times.$
    \item $\Aut(X)=\langle\sigma_{(123)},\sigma_{(23)}\rangle\simeq \fS_3$,
    and $X$ is given by
   \begin{align*}
    & x_1x_2x_3+x_4^2(x_1+x_2+x_3)+x_4f_2=0,\\
    \text{with}&\,\,f_2\in k[x_4, x_5] \text{ a general quadratic form};
    \end{align*}
\end{enumerate}
\end{prop}
\begin{proof}
 
We know that $\Aut(X)$ fits into the exact sequence 
\begin{align}\label{eq:3sinseq}
    0\to H\to\Aut(X)\stackrel{\rho}{\to}\fS_3,
\end{align}
where $H$ is the subgroup of $\Aut(X)$ fixing three singular points. 
Assume $\Aut(X)$ does not fix any singular point, i.e., there exists an element $\sigma_{123}\in\Aut(X)$ with $\rho(\sigma_{123})=(1,2,3)\in\fS_3$. Since $\sigma_{123}$ preserves the form \eqref{eq:3sin} and $q_1,q_2,q_3$ define at most $3$ points in $\bP^1_{x_4,x_5}$, we may assume that $\sigma_{123}$ takes the form
$$
\sigma_{123}: (\mathbf{x})\mapsto (s_1x_2,s_2x_3,s_3x_1,x_4,\zeta_3^r x_5),
$$
for $ r=0$ or $1$, and $s_1,s_2,s_3\in k^\times.$ The cyclic action, together with the torus action on $x_1, x_2, x_3$, imply that 
$$
q_2=\sigma_{123}^*(q_1),\quad\text{and} \quad q_3=\sigma_{123}^*(q_2).
$$
It follows that $s_1=s_2=s_3=\pm1.$ Now we discuss two cases of $r$:

If $r=0$, we may assume that $q_1=q_2=q_3=x_4^2$. Then $\Aut(X)$ contains a natural $\fS_3$-action, permuting the coordinates $x_1,x_2,x_3.$ It remains to classify possibilities of $H$. Assume that $H$ is nontrivial. A $\tau\in H$ acts diagonally on $x_1,\ldots,x_4$, since it preserves \eqref{eq:3sin}, and one can diagonalize $\tau$ without changing the form of \eqref{eq:3sin}. Thus, we may assume  that $\tau$ takes the form
$$
\tau: (\mathbf{x})\mapsto (x_1,x_2,x_3,a_1x_4,a_2x_5), \quad a_1=\pm1,\quad a_2\in k^\times.
$$
Recall that $f_3$ defines at most three points on $\bP^1_{x_4,x_5}$, preserved by $\tau$. We have the following cases: 
\begin{itemize}
\item $f_3$ defines three distinct points. It follows that $a_1=1$, $a_2=\zeta_3$, and $f_3=ax_4^3+x_5^3$, for some $a\in k$. In this case, $X$ has singularity type $3\sA_2$ and $\Aut(X)=C_3\times\fS_3$.
\item $f_3$ defines two distinct points, necessarily fixed by $\tau$. It follows that $f_3=ax_4x_5^2$, for some $a\in k^\times$. Note that $f_3=ax_4^2x_5$ would give nonisolated singularities on $X$, and we exclude this. Thus, we have $a_1=-1, a_2=\zeta_4$, $X$ has $3\sA_3$-singularities and $\Aut(X)=C_4\times\fS_3$.
\item $f_3$ defines one point, necessarily fixed by $\tau$. We have $f_3=x_5^3$, $X$ has $3\sA_2$-singularities and $\Aut(X)=C_6\times\fS_3$. 
\end{itemize}
If $f_3$ is a general cubic form, then $X$ has $3\sA_2$-singularities and $\Aut(X)=\fS_3$. If $f_3=x_4f_2$, where $f_2$ is a general quadratic form, then $X$ has $3\sA_3$-singularities and $\Aut(X)=\fS_3$. 

Now we consider the case $r=1$. Up to a change of variables, we may assume that
$$
q_1=(x_4+x_5)^2,\quad q_2=(x_4+\zeta_3 x_5)^2,\quad q_3=(x_4+\zeta_3^2 x_5)^2.
$$
Let $\sigma_{23}\in\Aut(X)$ be an element fixing $p_1$ and switching $p_2$ and $p_3.$ Then $\sigma_{23}$ also fixes the point in $\bP^1_{x_4,x_5}$ defined by $q_1$ and switches the points defined by $q_2, q_3$. The only possible such action on $\bP^1$ is switching the coordinates $x_4$ and $x_5$. But the points defined by $f_3$ need to be preserved by both $\sigma_{123}$ and $\sigma_{12}$. The only possibility is $f_3=a(x_4^3+x_5^3)$, for some $a\in k^\times$, and $X$ has $3\sA_3$-singularities. In particular, $\sigma_{23}$ takes the form 
$$
\sigma_{23}:(\mathbf{x})\mapsto(x_1,\zeta_3x_3,\zeta_3^2x_2,x_5,x_4).
$$
In the case of $3\sA_2$-singularities, $\sigma_{12}$ does not exists, i.e., $\rho(\Aut(X))=C_3.$ We then classify the possibilities of $H$. For any $\eta\in H$, $\eta$ fixes three singularities of $X$ in $\bP^4$ and three points in $\bP^1$ defined by $q_1,q_2,q_3$. One sees that $\eta$ acts on $\bP^4$ diagonally, with weights $(a_1,a_2,a_3,a_4,a_4)$. As above, we see that $f_3$ takes the following forms:
\begin{itemize}
\item $f_3$ defines three distinct points and $f_3=a(x_4^3+x_5^3)$, for some $a\in k^\times$. In this case $X$ has $3\sA_3$-singularities.
    \item $f_3$ defines three distinct points and $f_3=ax_4^3+bx_5^3$ for some $a\ne b\in k^\times$. In this case $X$ has $3\sA_2$-singularities.
    \item $f_3$ defines two distinct points, necessarily fixed by $\sigma_{123}$, i.e., $f_3=ax_4^2x_5$ or $ax_4x_5^2$, for some $a\in k^\times$. But in this case, $X$ is not $\langle\sigma_{123}\rangle$-invariant.
    \item $f_3$ defines one point, necessarily fixed by $\sigma_{123}$, i.e., $f_3=ax_4^3$ or $ax_5^3$, for some $a\in k^\times$. In this case, $X$ has $3\sA_2$-singularities.
\end{itemize}
It is not hard to check that in all cases above, $H$ is trivial, and $\Aut(X)=C_3$ when $X$ has $3\sA_2$-singularities; $\Aut(X)=\fS_3$ when $X$ has $3\sA_3$-singularities. 
\end{proof}

\begin{prop}
    Let $X$ be a cubic with $3\sA_2$ or $3\sA_3$-singularities. Then the following $G$-actions are not stably linearizable, for very general $X$ in the corresponding families in Proposition~\ref{prop:3Anauto}:
    \begin{enumerate}
        \item $X$ has $3\sA_2$-singularities and $\Aut(X)=\fS_3$: $G=\langle \sigma_{(123)}\rangle$,
          \item $X$ has $3\sA_2$-singularities and $\Aut(X)=C_3\times\fS_3$: $G=\langle \sigma_{(123)}\rangle$,
          \item $X$ has $3\sA_3$-singularities and $\Aut(X)=\fS_3$: $G=\langle \sigma_{(123)}\rangle$.
    \end{enumerate}
\end{prop}
\begin{proof}
    We use specialization, as in Proposition~\ref{prop:2A3noplanespec}. 
    By Proposition~\ref{prop:3Anauto}, 
    cubic threefolds in Case (1) are given by 
    $$
    x_1x_2x_3+x_4^2(x_1+x_2+x_3)+f_3=0,
    $$
    for a general cubic form $f_3\in k[x_4,x_5]$ or $f_3=x_4f_2$. We may assume that $f_3$ defines three distinct points in $\bP^1_{x_4,x_5}$, and is isomorphic to the cubic form $x_4^3+x_5^3$. Up to a change of variables, a very general cubic in Case (1) is a fiber of the family 
    \begin{align}\label{eq:3D4family1}
           \cX\to \bA_{s,t}^2,
    \end{align}
whose generic fiber is given by 
$$
    x_1x_2x_3+(sx_4+tx_5)^2(x_1+x_2+x_3)+x_4^3+x_5^3=0.
$$
The $\sigma_{(123)}$-action extends to the family $\cX$ and remains unchanged under the change of variables since it acts trivially on $x_4$ and $x_5$. The generic fiber of $\cX$ is a cubic with $3\sA_2$-singularities at
$$
p_1=[1:0:0:0:0],\quad p_2=[0:1:0:0:0],\quad p_3=[0:0:1:0:0].
$$
The special fiber $X_{0,0}$ above $s=t=0$ is a cubic with $3\sD_4$-singularities. The $\sigma_{(123)}$-action on $X_{0,0}$ is not stably linearizable, by Proposition~\ref{prop:3D4coho}. As in Proposition~\ref{prop:2A3noplanespec}, we apply specialization to a resolution of singularities of the generic fiber. This can be achieved by blowing up three sections of $\cX\to\bA^2_{s,t}$ corresponding to $p_1,p_2$ and $p_3$. After the blowup, the new family has smooth generic fiber, and the special fiber above $s=t=0$ has $BG$-rational singularities: it has $9\sA_1$-singularities forming three $\sigma_{(123)}$-orbits. As in \cite[Proposition 2.9]{CTZcub}, we conclude that a very general member in the family is not stably linearizable.

The same argument applies to cubic threefolds in Case (2) and (3) as they form a subfamily of Case (1), with the same $\sigma_{(123)}$-action. 
\end{proof}

\begin{exam}\label{exam:3A2Burn}
    Let $X$ be the cubic with $3\sA_2$-singularities and $\Aut(X)=C_6\times\fS_3$. The element $\eta_2^2$ fixes a cubic surface $S$ with $3\sA_1$-singularities in $X$, given by 
$$
x_1x_2x_3+(x_1+x_2+x_3)x_4^2=0,
$$
contributing to a symbol 
\begin{align}\label{eq:incomp3A2}
    (C_3,\fD_6\actsfromleft k(S),(1)).
\end{align}
The residual $\fD_6$-action on $S$ is realized as permutations of the coordinates $x_1,x_2,x_3$ and the $-1$ sign change on $x_4$. The standard Cremona transformation on $\bP^3$
$$
(x_1,x_2,x_3,x_4)\mapsto (\frac{1}{x_1},\frac{1}{x_2},\frac{1}{x_3},\frac{1}{x_4})
$$
birationally transforms $S$ into the smooth quadric surface 
$$
Q=\big\{x_1x_2+x_2x_3+x_3x_1+x_4^2=0\big\}\subset\mathbb{P}^3
$$
with the same $\fD_6$-action on the ambient $\bP^3$. This $\fD_6$-action on $Q$ is not birational to an action on a $\bP^1$-bundle over $\bP^1$, see \cite[Example 9.1]{TYZ-3}. Using the same argument as there, one sees that \eqref{eq:incomp3A2} is an incompressible symbol. This symbol cannot appear in classes of linear actions. It follows that the $C_6\times \fS_3$-action on $X$ is not linearizable.

The same argument applies to the cubic $X$ with $3\sA_3$-singularities and $\Aut(X)=C_4\times\fS_3$, given by
\begin{align*}
    x_1x_2x_3+x_4^2(x_1+x_2+x_3)+ax_4x_5^2=0,\quad a\in k^\times.
    \end{align*}
    For $G=\Aut(X)$-action, the element $\eta_3^2$ contributes to a symbol 
    \begin{align}\label{eq:3A3incomsymb}
            (\langle\eta_3^2\rangle,\fD_6\actsfromleft k(S),(1)),
    \end{align}    
    where $S$ is the same cubic surface carrying the same $\fD_6$-action as in the symbol \eqref{eq:incomp3A2}. As above, we see that the symbol \eqref{eq:3A3incomsymb} is also incompressible and the $\Aut(X)$-action on $X$ is not linearizable.
\end{exam}

\subsection*{Singularity Type $3\sD_4$} 
There is a unique such cubic threefold, see \cite[Theorem 5.4]{All}, given by
\begin{align*}
    X=\{x_1x_2x_3+x_4^3+x_5^3\}.
\end{align*}
By Proposition~\ref{prop:autd4}, we have
$$
\Aut(X)=\langle\tau_{a,b},\eta,\sigma_{(45)},\sigma_{(123)},\sigma_{(12)}\rangle\simeq(\bG_m^2(k)\times \fS_3)\rtimes \fS_3,
$$
with generators described in that proposition. 

\begin{exam}
\label{exam:3d4}
Let $G=\langle\sigma_{(123)},\sigma_{(45)}\sigma_{(12)}, \eta, \tau_{1,-1},\tau_{-1,1}\rangle\simeq C_3\rtimes\fS_4$.
The action is not linearizable: $G$ cannot act linearly and generically freely on $\bP^3$ as it does not have a faithful $4$-dimensional representation. 
\end{exam}

\subsection*{Cohomology} 
For the $3\sD_4$ case, we compute $\rH^1(G,\Pic(\tilde X))$ for finite subgroups $G\subset\Aut(X)$. The analysis is similar to the $9\sA_1$-cubic in \cite{CTZcub}. Recall that the defect $\sigma(X)=4$ and $\mathrm{rk}\, \mathrm{Cl}(X)=5$. In particular, $\mathrm{Cl}(X)$ is generated by the nine planes in $X$:
$$
\Pi_{i,j}=\{x_i=x_4+\zeta_3^jx_5=0\},\quad i=1,2,3,\quad j=1,2,3\quad \zeta_3=e^{\frac{2\pi i}{3}},
$$
subject to relations
\begin{align}\label{eqn:3D4rel}
\sum_{i=1}^3\Pi_{i,j}=F, \quad\text{for }j=1,2,3,\quad\sum_{j=1}^3\Pi_{i,j}=F, \quad\text{for }i=1,2,3,
\end{align}
where $F$ denotes the class of the hyperplane section on $X$. 

\begin{prop}\label{prop:3d4H1}
      Let $X$ be the cubic of singularity type $3\sD_4$ and
      $$
      \sigma_{(123)}: \mathbf{(x)}\mapsto (x_3,x_1,x_2,x_4,x_5).
      $$
    Then 
    $$
    \rH^1(\langle \sigma_{(123)}\rangle,\Pic(\tilde{X})) =\bZ/3.
    $$
\end{prop}
\begin{proof}
    Using the generators and relations of $\mathrm{Cl}(X)$ described above, one can compute 
    $$
  \rH^1(\langle \sigma_{(123)}\rangle,\Cl({X})) =\bZ/3.
    $$
   Let $\tilde X\to X$ be an $\Aut(X)$-equivariant resolution of singularities via successive blowups of the singular points, and $E_i, i=1,\ldots,r$ the corresponding exceptional divisors. Since $\sigma_{(123)}$ acts transitively on the $3\sD_4$-points, it permutes the exceptional divisors without leaving any one fixed. Then 
   $$
\rH^2(\langle\sigma_{(123)}\rangle,\oplus_{i=1}^{r}\bZ\cdot E_i)=0.
   $$
   Using Lemma~\ref{lemm:h1seq}, we conclude 
   $$
    \rH^1(\langle \sigma_{(123)}\rangle,\Pic(\tilde{X}))=\rH^1(\langle \sigma_{(123)}\rangle,\Cl({X})) =\bZ/3.
   $$
\end{proof}

\subsection*{Linearizability}

\begin{prop}\label{prop:3d4linear}
Let $X$ be the cubic threefold with $3\sD_4$-singularities and $H=\langle\tau_{ab},\eta\sigma_{(123)}, \sigma_{(45)}\sigma_{(12)}\rangle$, from Proposition~\ref{prop:autd4}. For any subgroup $H'\subset H$, the $H'$-action on $X$ is linearizable if and only if each abelian subgroup of $H'$ fixes a point on $X$.
\end{prop}

\begin{proof}
Note that the existence of a fixed point on $X$ by each abelian subgroup is a necessary condition for linearizability. We show it is also sufficient in this case. Under the rational map 
$
\rho: X\dashrightarrow \bP^1\times\bP^1\times\bP^1\times\bP^1
$
given by 
$$
\mathbf{(x)}\mapsto (-x_3,x_4+x_5)\times(-x_1,\zeta_3^2x_4+\zeta_3x_5)\times(-x_2,\zeta_3x_4+\zeta_3^2x_5)\times(x_4,x_5),
$$
$X$ is birationally transformed to $S\times\bP^1$ where $S$ is a smooth del Pezzo surface of degree 6, realized as 
$$
\{u_1v_1w_1=u_2v_2w_2\}\subset\bP^1_{u_1,u_2}\times\bP^1_{v_1,v_2}\times\bP^1_{w_1,w_2}.
$$
The map $\rho$ is $H$-equivariant. The $H$-action on $S\times\bP^1$ is faithful on the factor $S$: $H$ acts on $S$ via the $\fS_3$-permutation of three copies of $\bP^1$ and the 2-dimensional torus action. The $H$-action on the $\bP^1$ in $S\times\bP^1$ factors through $\fS_3$. Observe that the $H$-action on $S$ is also birational to the $H$-action on $\bP^2$ via permutation of coordinates and the standard $\bG_m^2$ torus action on $\bP^2.$ Thus, the $H$-action on $S\times\bP^1$ is birational to the corresponding action on $\bP^2\times \bP^1$.
This action is birational to an action on a rank-1 vector bundle over $\bP^2$. By the no-name lemma, this is birational to an action on $\bA^1\times \bP^2$, with trivial action on the first factor and generically free action on the second factor. For any subgroup $H'\subset H$, if abelian subgroup of $H'$ fixes a point, the $H'$-action on $\bP^2$ is linear. It follows that the $H'$-action is linearizable, again by the no-name Lemma.  
\end{proof}

\begin{prop}\label{prop:3D4coho}
    Let $X$ be the cubic threefold with $3\sD_4$-singularities and $G\subseteq \Aut(X)$. Then the $G$-action on $X$ is (stably) linearizable if and only if $G$ does not contain an element conjugate to $\sigma_{(123)}$, and each abelian subgroup of $G$ fixes a point on $X$.


\end{prop}
\begin{proof}
    When an abelian subgroup of $G$ does not fix any point on $X$, the action is not stably linearizable. When $G$ contains an element conjugate to $\sigma_{(123)}$, the $G$-action on $X$ has an {\bf (H1)}-obstruction and is not stably linearizable, by Proposition~\ref{prop:3d4H1}. 
    
    Conversely, when $G$ does not contain an element conjugate to $\sigma_{(123)}$ and $G$ does not fix any singular point, then
    $$
    G\subset H=\langle\tau_{ab},\eta\sigma_{(123)}, \sigma_{(45)}\sigma_{(12)}\rangle\simeq \bG^2_m(k)\rtimes\fS_3.
    $$ 
   From Proposition~\ref{prop:3d4linear}, we see that the $G$-action on $X$ is linearizable when each abelian subgroup of $G$ fixes a point on $X$. If $G$ fixes a singular point, it is clearly linearizable. 
\end{proof}

\section{Four singular points}
\label{sect:four}

With our assumptions, the possible combinations of singularities, with specializations, are:

\centerline{
\xymatrix{
2\sA_2+2\sA_1 \ar[d]   \ar[r]  & 2\sA_3+2\sA_1  \ar[r] &  2\sD_4+2\sA_1 \\
4\sA_2 & 
}
}

In all cases, the singularities are in linear general position - indeed, if four singularities are contained in a plane, they must be $4\sA_1$, treated in \cite{CTZcub}.
We can thus assume that the singularities are at
$$
    p_1=[1:0:0:0:0], \quad p_2=[0:1:0:0:0],
    $$
    $$
    p_3=[0:0:0:1:0], \quad p_4=[0:0:0:1:0].
$$
Then $X$ is given by 
\begin{multline}\label{eq:4sineq}
    t_1x_1x_2x_3 + t_2x_1x_2x_4 + t_4x_1x_3x_4 + t_8x_2x_3x_4+\\
    +t_{15}x_5^3+ x_5^2(t_7x_1+ t_{11}x_2+ t_{13}x_3 + t_{14}x_4)+ \\
    +x_5(t_3x_1x_2 +  t_5x_1x_3+ t_6x_1x_4 + t_9x_2x_3  + t_{10}x_2x_4 + t_{12}x_3x_4) =0,
\end{multline}
for some $t_1,\ldots,t_{15}\in k.$ Up to a change of variables by torus actions, we may assume that $t_i=0$ or $1$, for $i=1,2,4,8.$  In the remaining of this section, we assume $p_1$ and $p_2$ are of the same singularity type and $p_3$ and $p_4$ are of the same type. Then there is an exact sequence 
$$
0\to H\to \Aut(X)\stackrel{\rho}{\to} \fS_4,
$$
and, except for $4\sA_2$, the image of $\rho$ is at most $C_2^2=\langle(1,2), (3,4)\rangle$.
\subsection*{Singularity Type $2\sA_2+2\sA_1$.} 
Assume that $p_1,$ $p_2$ are $\sA_2$-points.

\begin{prop}\label{prop:2A2+2A1class}
    If $X$ is a cubic threefold with $2\sA_2+2\sA_1$-singularities then: 
     \begin{itemize}
         \item Up to isomorphism, $X$ can be given by
   \begin{multline}\label{eq:2A2+2A1eq}
  x_1x_2x_3 + x_1x_2x_4+x_1x_3x_4+x_2x_3x_4+x_5^3+\\
     +x_5^2(a_1x_3 + a_2x_4-\frac14a_3^2(x_1+x_2))+\\
    +x_5(a_3(x_1x_3+x_2x_4)+a_4x_3x_4)=0,    
   \end{multline}
    for general $a_1,a_2,a_3,a_4\in k$.
    \item If $\Aut(X)$ does not fix any singular point, then 
one of the following holds
    $$
    \Aut(X)=\begin{cases}
        \langle\sigma_{(12)(34)},\sigma_{(12)}\rangle\simeq C_2^2,& \text{ when } a_1=a_2,\,\, a_3=0,\\
         \langle\sigma_{(12)(34)}\rangle\simeq C_2,& \text{ when } a_1=a_2,\,\, a_3\ne 0,
    \end{cases}
    $$
    where
    {\small
    \begin{align*}
        \sigma_{(12)(34)}&:(\mathbf{x})\mapsto(x_2,x_1,x_4,x_3,x_5),\\
        \sigma_{(12)}&:(\mathbf{x})\mapsto (x_2,x_1,x_3,x_4,x_5).
    \end{align*}}
     \end{itemize}
\end{prop}

\begin{proof}
Since $p_1$ and $p_2$ are $\sA_2$-singularities, the quadratic terms after $x_1$, respectively $x_2$, define two quaternary quadratic forms of rank $3$. This gives two nonlinear conditions on the parameters in \eqref{eq:4sineq}:
{\small
   \begin{multline}
       \label{eq:2A22A1cond1}
       t_{1}^2t_{6}^2 + 4t_{1}t_{2}t_{4}t_{7} - 2t_{1}t_{2}t_{5}t_{6} - 2t_{1}t_{3}t_{4}t_{6} + t_{2}^2t_{5}^2 - 2t_{2}t_{3}t_{4}t_{5} + 
        t_{3}^2t_{4}^2=0,\\
    t_{1}^2t_{10}^2 + 4t_{1}t_{2}t_{8}t_{11} - 2t_{1}t_{2}t_{9}t_{10} - 2t_{1}t_{3}t_{8}t_{10} + t_{2}^2t_{9}^2 - 2t_{2}t_{3}t_{8}t_{9}
        + t_{3}^2t_{8}^2=0.
   \end{multline}
}
When $t_1=t_2=0$, $X$ has nonisolated singularities. Thus, 
up to a change of variables, one may assume that
$$
t_1=1,\quad t_3=t_5=t_9=0,
$$
which reduces \eqref{eq:2A22A1cond1} to 
\begin{align}\label{eq:2A22A1cond2}
     4t_2t_4t_7 + t_6^2=4t_2t_8t_{11} + t_{10}^2=0.
\end{align}
It follows that $t_2\ne 0$ since otherwise $t_2=t_6=t_{10}=0$ and $X$ has nonisolated singularities. Similarly, one may check $t_4, t_8\ne 0$, since otherwise it introduces $\sA_3$-singularities. Hence,  $t_1=t_2=t_4=t_8=1.$ 
Up to a change of variables, we may assume $t_6=0, $ simplifying \eqref{eq:2A22A1cond2} as 
$
t_7=4t_{11}+t_{10}^2=0.
$
One can also check $t_{15}\ne 0$, since otherwise $X$ has $2\sA_1+2\sA_3$-singularities. Thus, we may put $t_{15}=1$, by scaling $x_5$, and the equation of $X$ is of the form \eqref{eq:2A2+2A1eq}.

Now assume that $\Aut(X)$ does not fix any singular point, i.e., there exists $\sigma\in \Aut(X)$ such that $\rho(\sigma)=(1,2)(3,4)$ and $\sigma$ takes the form 
$$
\sigma:(\mathbf{x})\mapsto(s_2x_2+r_2x_5,s_1x_1+r_1x_5,s_4x_4+r_4x_5,s_3x_3+r_3x_5,x_5)
$$
for $s_1,\ldots,s_4\in k^\times$ and $r_1,\ldots,r_4\in k$. The fact that $\sigma$ leaves $X$ invariant leads to a system of equations in the parameters $s_1,\ldots,s_4,r_1,\ldots,r_4$. Solving the system, we find that such an element $\sigma$ exists if and only if
$$
a_1=a_2,\quad
s_1=s_2=s_3=s_4=1, \quad  r_1=r_2=r_3=r_4=0.
$$
Using the same method, we find that when $a_1=a_2$, an element $\tau\in\Aut(X)$ with $\rho(\tau)=(1,2)$ exists if and only if $a_3=0$, and 
$$
\tau: (\mathbf{x})\mapsto (x_2,x_1,x_3,x_4,x_5).
$$
Moreover, any $h\in H$ fixing four singular points is trivial. 
\end{proof}


\begin{prop}\label{prop:2a22a1spec2d42a1}
    Let $X$ be a very general cubic threefold  with $2\sA_2+2\sA_1$-singularities, given by \eqref{eq:2A2+2A1eq} with $a_1=a_2$. Then the $\langle\sigma_{(12)(34)}\rangle$-action on $X$,
    specified in Proposition~\ref{prop:2A2+2A1class},
    is not stably linearizable.
\end{prop}

\begin{proof}
    We use the notation from Proposition~\ref{prop:2A2+2A1class}. Let $a=a_4^{-1/4}$. Under the change of variables
      $$
    y_1=ax_1,\quad y_2=ax_2,\quad y_3=\frac{1}{a^2}x_3,\quad x_4=\frac{1}{a^2}x_4,\quad y_5=x_5,
    $$
    the equation \eqref{eq:2A2+2A1eq} becomes
    \begin{multline}\label{eq:2A2+2A1}
        x_1x_2x_3 + x_1x_2x_4 + a^3x_1x_3x_4 + a^3x_2x_3x_4 + a_3ax_5(x_1x_3 + 
    x_2x_4)+\\ + x_3x_4x_5 + x_5^2( 
    a_2a^2(x_3 + x_4)-\frac{a_3^2}{4a}(x_1 +x_2)) + x_5^3.
    \end{multline}
    For fixed $a_2,\lambda\in k$, we can consider the 1-parameter family of cubic $\cX\to\bA_a$ parameterized by $a$ given by \eqref{eq:2A2+2A1} with $a_2$ and $a_3=\lambda a$. In particular, generic fiber of $\cX$ is given by 
    \begin{multline*}
        x_1x_2x_3 + x_1x_2x_4 + a^3x_1x_3x_4 + a^3x_2x_3x_4 + a^2x_5(x_1x_3 + 
    x_2x_4)+\\ + x_3x_4x_5 + x_5^2( 
    a_2a^2(x_3 + x_4)-\frac{a}{4}(x_1 +x_2)) + x_5^3.
    \end{multline*}
    The $\langle\sigma_{(12)(34)}\rangle$-action naturally extends to $\cX$ and is not stably linearizable on the special fiber $X_0$ above $a=0$: $X_0$ has $2\sD_4+2\sA_1$-singularities and has an {(\bf H1)}-obstruction by Proposition~\ref{prop:2D4+2A1H1}.
    
Similarly as before, e.g., in Proposition~\ref{prop:2A4spec2A5}, we apply specialization to the resolution of singularities of the generic fiber in the family $\cX$ to conclude the $\langle\sigma_{(12)(34)}\rangle$-action on a very general member in the family $\cX$, i.e., a very general cubic with $2\sA_2+2\sA_1$-singularities, is not stably linearizable.

\end{proof}

\subsection*{Singularity Type $2\sA_3+2\sA_1$.}  Assume $p_1,$ $p_2$ are $\sA_3$-points. 
Recall from Lemma~\ref{lemm:2A3mA1defect} that a cubic threefold $X$ with $2\sA_3+2\sA_1$-singularities can have $d(X)=1$ or $2$. 

\begin{lemm}
\label{lemm:d1}
        Let $X$ be a cubic threefold with $2\sA_3+2\sA_1$-singularities and $d(X)=1$. Then the $\Aut(X)$-action on $X$ is linearizable. 
\end{lemm}
\begin{proof}
    Let $G=\Aut(X)$. If $d(X)=1$ there is a unique, necessarily $G$-invariant plane $\Pi$ contained in $X$. There are two possibilities: either $\Pi$ contains both $\sA_3$-points, or only one. In the first case, $X$ contains an $G$-invariant line that is disjoint from the plane, namely the line between the two $\sA_1$-points. The action is then linearizable by \cite[Lemma 1.1]{CTZ}.
    In the second case, the $\sA_3$-point contained in $\Pi$ is fixed.
\end{proof}

Thus we focus on the $d(X)=2$ case.

\begin{prop}\label{prop:2A32A1}
    Let $X$ be a cubic with $2\sA_3+2\sA_1$-singularities and defect $d(X)=2$. Then, up to isomorphism, $X$ 
    is given by 
        \begin{align}
 \label{eq:2A3+2A1eq2}
x_1x_2(x_3+x_4)+x_5^2(x_1+x_2+a_1x_3+a_1x_4)+x_3x_4x_5+a_2x_5^3
=0,
  \end{align}
   for general $a_1,a_2\in k$,
    $$
    \Aut(X)=
    \begin{cases}
    \langle\sigma_{(12)},\sigma_{(12)(34)}, \eta_1\rangle\simeq C_2\times C_6,&\text{when } a_1=a_2=0,\\
        \langle\sigma_{(12)},\sigma_{(12)(34)}\rangle\simeq C_2^2,&\text{otherwise},
    \end{cases}
    $$
    where 
    \begin{align*}
        \sigma_{(12)}&: (\mathbf{x})\mapsto(x_2,x_1,x_3,x_4,x_5),\\
        \sigma_{(12)(34)}&: (\mathbf{x})\mapsto(x_2,x_1,x_4,x_3,x_5),\\
        \eta_1&:(\mathbf{x})\mapsto(\zeta_3x_1,\zeta_3x_2,\zeta_3^2x_3,\zeta_3^2x_4,x_5).
    \end{align*}
\end{prop}

\begin{proof}
    Following the proof of Proposition~\ref{prop:2A2+2A1class}, we know that up to change of variables, the parameters in \eqref{eq:4sineq} satisfy
    $$
    t_1=t_2=1,\quad t_3=t_5=t_9=4t_4t_7 + t_6^2=4t_8t_{11} + t_{10}^2=0.
    $$
    When the defect $d(X)=2$, by Proposition~\ref{lemm:2A3mA1defect}, we know that $X$ contains three planes, and two of them are spanned by
    $$
    \Pi_1=\langle p_1,p_3,p_4\rangle=\{x_2=x_5=0\},\quad  \Pi_2=\langle p_2,p_3,p_4\rangle=\{x_1=x_5=0\}.
    $$
    This implies $t_4=t_8=0$, and thus $t_6=t_{10}=0.$ Then up to a change of variables, 
we obtain the desired form \eqref{eq:2A3+2A1eq2}.

Using the same method as in the proof of Proposition~\ref{prop:2A2+2A1class}, one can find all possibilities of $\Aut(X)$ as stated in both cases.
\end{proof}

\begin{prop}\label{prop:2a32a1spec2d42a1}
    Let $X$ be a very general cubic with $2\sA_3+2\sA_1$-singularities and $d(X)=2$. Then the $\langle\sigma_{(12)(34)}\rangle$-action on $X$ from Proposition~\ref{prop:2A32A1},
    is not stably linearizable.
\end{prop}
\begin{proof}
    By Proposition~\ref{prop:2A32A1}, we know that all such cubics are given by \eqref{eq:2A3+2A1eq2}. Let $a=a_2^{-\frac{1}{12}}$. Under the change of variables 
    $$
    y_1=\frac1ax_1,\quad y_2=\frac1ax_2,\quad y_3=a^2x_3,\quad y_4=a^2x_4,\quad y_5=a^4x_5,
    $$
    the equation \eqref{eq:2A3+2A1eq2} becomes 
\begin{align}\label{eq:2A32A1param}
y_1y_2(y_3+y_4)+y_3y_4y_5+x_5^2(a^9(x_1+x_2)+a_1a^6(x_3+x_4)).
\end{align}
For any fixed $\lambda\in k$, we may consider all the cubic threefolds given by \eqref{eq:2A32A1param} with  $a_1a^6=\lambda$ as a 1-parameter family parameterized by $a$:
$$
\cX\to \bA^1_{a},
$$
where the general fiber of $\cX$ is a cubic threefold with $2\sA_3+2\sA_1$-singularities given by 
$$
y_1y_2(y_3+y_4)+y_3y_4y_5+x_5^2(a^9(y_1+y_2)+\lambda(y_3+y_4))=0.
$$
The special fiber $X_{0}=\cX_{0}$ above $a=0$ has $2\sD_4+2\sA_1$-singularities. The $\sigma_{(12)(34)}$-action on $X_0$ is not stably linearizable, see Proposition~\ref{prop:2D4+2A1H1}. 

 Applying specialization to a resolution of singularities of the generic fiber, we conclude that the $\langle\sigma_{(12)(34)}\rangle$-action on a very general member in the family $\cX$ is not stably linearizable.
\end{proof}
\subsection*{Singularity Type $2\sD_4+2\sA_1$.} 
Assume that $p_1,$ $p_2$ are $\sD_4$-points.
\begin{prop}\label{prop:2D4+2A1}
      Let $X$ be a cubic with $2\sD_4+2\sA_1$-singularities. Then, 
      up to isomorphism, $X$ is given by
   \begin{align}\label{eq:2D4+2A1eq}
  x_1x_2x_3 + x_1x_2x_4+a_1x_3x_4x_5
    + x_5^2(a_2x_3 + a_2x_4)
    + x_5^3=0,
\end{align}
for general $a_1,a_2\in k$, and
    $$
    \Aut(X)=\begin{cases}
\langle\tau_a,\eta_2,\sigma_{(12)},\sigma_{(12)(34)}\rangle\simeq(\bG_m(k)\times C_2)\rtimes C_2^2,&\text{ when }a_2=0,\\
        \langle\tau_a,\sigma_{(12)},\sigma_{(12)(34)}\rangle\simeq \bG_m(k)\rtimes C_2^2,&\text{ otherwise},
    \end{cases}
    $$
    where 
    \begin{align*}
        \tau_a:& (\mathbf{x})\mapsto (ax_1, a^{-1}x_2, x_3, x_4, x_5),\quad a\in k^\times,\\
         \eta_2:& (\mathbf{x})\mapsto (x_1, -x_2, -x_3, -x_4, x_5),\\
          \sigma_{(12)}:& (\mathbf{x})\mapsto (x_2, x_1, x_3, x_4, x_5),\\
           \sigma_{(12)(34)}:& (\mathbf{x})\mapsto (x_2, x_1, x_4, x_3, x_5).
    \end{align*}
    
\end{prop}

\begin{proof}
Existence of $\sD_4$-points $p_1$ and $p_2$ implies that the quadratic terms after $x_1$, respectively $x_2$, define two quadratic forms in 4 variables of rank $2$. This imposes a system of nonlinear conditions on the parameters $t_1,\ldots,t_{15}$. Solving the system via {\tt magma}, and excluding the components of the solutions whose general members define a cubic with nonisolated singularities, we find the conditions on parameters:
\begin{multline}\label{eq:2D4+2A1condition}
        t_4=t_8=0, \,\,\,
  t_1t_6=t_2t_5,\,\,\,
        t_1t_7=t_3t_5,\,\,\,
        t_1t_{10} =t_2t_9,\,\,\,
        t_1t_{11}= t_3t_9,\\
        t_2t_7 =t_3t_6,\,\,\,
        t_2t_{11} = t_3t_{10},\,\,\,
        t_5t_{10} =t_6t_9,\,\,\,
        t_5t_{11}= t_7t_9,\,\,\,
        t_6t_{11}= t_7t_{10}.
\end{multline}
Up to a change of variables, we may assume $t_1=1, t_3=t_5=t_9=0$, reducing \eqref{eq:2D4+2A1condition} to
$
    t_6=t_7=t_{10}=t_{11}=0.
$
One may check $t_2\ne 0$, since otherwise $X$ has nonisolated singularities, and $X$ is of the form 
   \begin{align*}
  x_1x_2x_3 + x_1x_2x_4+t_{15}x_5^3
    + x_5^2(t_{13}x_3 + t_{14}x_4)
    +t_{12}x_3x_4x_5 =0.
\end{align*}
Up to a change of variables, we may assume that $t_{13}=t_{14}$ and $t_{15}=1$.
 
To find $\Aut(X)$, we first observe that $\langle \sigma_{(12)(34)},\sigma_{(12)}\rangle\subset\Aut(X)$ as specified in the assertion. So it suffices to classify $g\in \Aut(X)$ fixing all four singular points. Such elements take the form 
 $$
 g:(\mathbf{x})\mapsto (s_1x_1+r_1x_5,s_2x_2+r_2x_5,s_3x_3+r_3x_5,s_4x_4+r_4x_5 x_5),
 $$
 for some $s_1,\ldots,s_4\in k^\times$ and $r_1,\ldots,r_4\in k$. Let $f$ be the equation~\eqref{eq:2D4+2A1eq}. As before, $g^*(f)=s_1s_2s_3f$ imposes a system of equations on the parameters. Solving this system leads to the assertions about $\Aut(X)$.
\end{proof}

\subsection*{Cohomology}From Proposition~\ref{prop:defectd4pt}, we know that $\Cl(X)$ is generated by the five planes in $X$. Using \eqref{eq:2D4+2A1eq}, one finds their equations:
$$
    \Pi_1=\{x_1=x_5=0\},\quad  \Pi_2=\{x_2=x_5=0\}, \quad  
    $$
    $$
    \Pi_3=\{x_3+x_4=x_5+\sqrt{a_1}x_3=0\},
$$
$$
\Pi_4=\{x_3+x_4=x_5=0\}, \quad  \Pi_5=\{x_3+x_4=x_5-\sqrt{a_1}x_3=0\}.
$$
The class group $\Cl(X)$ is generated by $\Pi_1,\ldots,\Pi_5$, with relation 
$$
\Pi_1+\Pi_2=\Pi_3+\Pi_5.
$$
The involution $\sigma_{(12)(34)}\sigma_{(12)}$ and $\eta_2$ both switch $\Pi_3$ and $\Pi_5$ and leave other planes invariant, while  $\sigma_{(12)}$ switches $\Pi_1$ and $\Pi_2$ and leaves other planes invariant.    

\begin{prop}
\label{prop:2D4+2A1H1}
    With notation as above, one has 
    $$
    \rH^1(\langle\sigma_{(12)(34)}\rangle,\Pic(\tilde X))=\bZ/2.
    $$
\end{prop}
\begin{proof}
    Choose a basis of $\Cl(X)$ consisting of the classes
    $$
    \Pi_3,\quad \Pi_1+\Pi_2-\Pi_3, \quad \Pi_4,\quad\Pi_2-\Pi_3.
    $$ 
    The involution $\sigma_{(12)(34)}$ switches the first two elements, fixes the third one, and acts on $\Pi_2-\Pi_3$ via (-1) multiplication. This implies 
    $$
      \rH^1(\langle\sigma_{(12)(34)}\rangle,\Cl(X))=\bZ/2.
    $$
    Since  $\sigma_{(12)(34)}$ does not fix any singular points, it does not fix any class of exceptional divisors $E_i$ of the resolution of singularities. In particular, 
    $$
      \rH^2(\langle\sigma_{(12)(34)}\rangle,\oplus_i E_i)=0.
    $$
    By Lemma~\ref{lemm:h1seq}, we conclude 
 $$
      \rH^1(\langle\sigma_{(12)(34)}\rangle,\Pic(\tilde X))=\rH^1(\langle\sigma_{(12)(34)}\rangle,\Cl(X))=\bZ/2.
    $$
    
\end{proof}

\subsection*{Linearizability}
\begin{coro}
\label{coro:2D4+2A1}
Let $X$ be a cubic threefold with $2\sD_4+2\sA_1$-singularities and $G\subseteq \Aut(X)$. 
The $G$-action on $X$ is not (stably) linearizable if and only if $G$ contains an element conjugate to $\sigma_{(12)(34)}$.
\end{coro}
\begin{proof}
 If $G$ contains an element conjugate to $\sigma_{(12)(34)}$. By Proposition~\ref{prop:2D4+2A1H1}, we know the $G$-action on $X$ has an {\bf (H1)}-obstruction and is not stably linearizable. Conversely, assume $G$ does not contain such an element and $G$ does not fix any singular point. From the classification in Proposition~\ref{prop:2D4+2A1}, we are in the case when $a_2=0$ and up to conjugation, $G\subseteq G'$ where $G'$ is the group generated by $\tau_a,\sigma_{(12)}$ and $\eta_2\sigma_{(12)(34)}$.  One can then check that $G'$ leaves invariant the plane $\Pi_3$ and the line $\{x_1=x_2=x_5=0\}\subset X$ disjoint from $\Pi_3$. It follows that the $G$-action on $X$ is linearizable, as in \cite[Lemma 1.1]{CTZ}.
\end{proof}

\subsection*{Singularity Type $4\sA_2$.}  Assume that $p_1,  p_2, p_3, p_4$ are $\sA_2$-points on $X$. We start with a classification of actions and normal forms.

\begin{prop}
\label{prop:4A2-newform}
    Let $X$ be a cubic threefold with $4\sA_2$-singularities. Then, up to isomorphism, $X$ is given by 
  \begin{multline}\label{eq:4A2newform}
x_1x_2x_3+x_1x_2x_4+x_1x_3x_4+x_2x_3x_4+x_5^3+ax_5^2(x_1+x_2+x_3+x_4)+\\
        +x_5(r_{1}(x_1x_2+x_3x_4)+r_2(x_1x_3+x_2x_4)+r_3(x_1x_4+x_2x_3))=0, 
    \end{multline}
for general $r_1,r_2,r_3\in k$, and 
$$
a=-(\frac14r_1^2 - \frac12r_1r_2 - \frac12r_1r_3 + \frac14r_2^2 - 
    \frac12r_2r_3 + \frac14r_3^2),
    $$
    with
    $$
\Aut(X)=\begin{cases}
   \langle\eta_3,\sigma_{(12)},\sigma_{(1234)}\rangle\simeq C_3\times\fS_4,&\text{when } r_1=r_2=r_3=0,\\
    \langle\sigma_{(12)},\sigma_{(1234)}\rangle\simeq\fS_4,&\text{when } r_1=r_2=r_3\ne0,\\
     \langle\eta_3^2\sigma_{(234)},\sigma_{(12)(34)}
        \rangle\simeq\fA_4,&\text{when } r_1=\zeta_3r_2=\zeta_3^2r_3\ne 0,\\
          \langle\sigma_{(12)},\sigma_{(12)(34)}\rangle\simeq\fD_4,&\text{when } r_2=r_3,\\
           \langle\sigma_{(13)(24)},\sigma_{(12)(34)}\rangle\simeq C_2^2,&\text{otherwise,}
\end{cases}
$$
where 
\begin{align*}
    \eta_3&: (\mathbf{x})\mapsto(x_1,x_2,x_3,x_4,\zeta_3x_5),\\
    \sigma_{(12)}&: (\mathbf{x})\mapsto(x_2,x_1,x_3,x_4,x_5),\\
     \sigma_{(1234)}&: (\mathbf{x})\mapsto(x_2,x_3,x_4,x_1,x_5),\\     \sigma_{(234)}&: (\mathbf{x})\mapsto(x_1,x_3,x_4,x_2,x_5),\\  
       \sigma_{(12)(34)}&: (\mathbf{x})\mapsto(x_2,x_1,x_4,x_3,x_5),\\     
        \sigma_{(13)(24)}&: (\mathbf{x})\mapsto(x_3,x_4,x_1,x_2,x_5).
\end{align*}
\end{prop}

\noindent

\begin{proof}
Four $\sA_2$-points impose the system of equations
    \begin{multline}\label{eq:4A2cond}
t_1^2t_{10}^2 + 4t_1t_2t_8t_{11} - 2t_1t_2t_9t_{10} - 
        2t_1t_3t_8t_{10} + t_2^2t_9^2 - 2t_2t_3t_8t_9 + t_3^2t_8^2\\
    =t_1^2t_{12}^2 + 4t_1t_4t_8t_{13} - 2t_1t_4t_9t_{12} - 
        2t_1t_5t_8t_{12} + t_4^2t_9^2 - 2t_4t_5t_8t_9 + t_5^2t_8^2\\
   = t_2^2t_{12}^2 + 4t_2t_4t_8t_{14} - 2t_2t_4t_{10}t_{12} - 
        2t_2t_6t_8t_{12} + t_4^2t_{10}^2 - 2t_4t_6t_8t_{10} + t_6^2t_8^2\\
     =t_1^2t_6^2 + 4t_1t_2t_4t_7 - 2t_1t_2t_5t_6 - 2t_1t_3t_4t_6
        + t_2^2t_5^2 - 2t_2t_3t_4t_5 + t_3^2t_4^2=0
    \end{multline}
    on parameters $t_1,\ldots,t_{15}$ in \eqref{eq:4sineq}. At least two of $t_1, t_2, t_4, t_8$ are nonzero, since otherwise $X$ has $3\sA_1+\sD_4$-singularities. Up to a change of variables, we may assume $t_1=t_2=1$, 
    $t_3=t_5=t_9=0$. If $t_4=0$, \eqref{eq:4A2cond} implies $t_6=t_{12}=0$ and $X$ has nonisolated singularities. Hence $t_4=1$ and the same argument shows $t_8=1.$ 
   Then we may assume $t_6=0$ reducing the system \eqref{eq:4A2cond} to 
    \begin{align*}
            t_7=t_{10}^2 + 4t_{11}=t_{12}^2 + 4t_{13}=
    (t_{10}-t_{12})^2+4t_{14}=0.
    \end{align*} 
    One may check that a general solution defines a cubic with $4\sA_2$-points. After a change of variables we obtain \eqref{eq:4A2newform}; the automorphisms $\Aut(X)$ can be classified using an argument similar to that in Proposition~\ref{prop:2A2+2A1class}.
\end{proof}



\begin{prop}\label{eq:4A2->2A5}
    Let $X$ be a very general cubic with $4\sA_2$-singularities given by \eqref{eq:4A2newform}, with $r_2=r_3$. Then $X$ is not $\langle\sigma_{(13)(24)}\rangle$-stably linearizable.
\end{prop}
\begin{proof}
    When $r_2=r_3$, one may assume that $r_1=0$ up to isomorphisms. Under the change of variables
    $$
    y_1=-r_2^{1/2}(x_1+x_2),\quad    y_2=-r_2^{1/2}(x_3+x_4), \quad y_3=x_5, 
    $$
    $$
    y_4=-r_2^{-1/4}x_4,\quad y_5=-r_2^{-1/4}x_2,
    $$
    equation \eqref{eq:4A2newform} becomes 
    \begin{align}\label{eq:4A22S5fam}
           y_1y_2y_3 + y_2y_5^2 + y_1y_4^2+ y_3^3 -r_2^{-3/4}y_1y_2y_4 -r_2^{-3/4}y_1y_2y_5=0.
    \end{align}
    Let $r=-r_2^{-3/4}$, one may view \eqref{eq:4A22S5fam} as a family of cubic threefolds 
    $
    \cX\to\bA^1_{r}
    $
    parameterized by $r\in k$. The $\sigma_{(13)(24)}$-action extends to the family $\cX$ and takes the form
    $$
   \iota:\mathbf{(y)}\mapsto (y_2,y_1,y_3,y_5,y_4).
    $$
    The generic fiber of the family  has $4\sA_1$-singularities at 
    $$
    p_1=[1:0:0:0:0],\quad p_2=[1:0:0:0:-r],
    $$
    $$
    p_3=[0:1:0:0:0],\quad p_4=[0:1:0:-r:0],
    $$
    and the special fiber $X_0$ when $r=0$ has $2\sA_5$-singularities. The action of $\langle\sigma_{(13)(24)}\rangle$ on $X_0$ is not stably linearizable, by Proposition~\ref{prop:coho}. To resolve the singularities in the generic fiber, one can first equivariantly blow up two disjoint sections passing through the singular points $p_1$ and $p_3$ respectively, and then blow up those passing through $p_2$ and $p_3$, respectively. The resulting family has smooth generic fiber and the special fiber above $a=0$ has $BG$-rational singularities: it has $2\sA_1$-singularities in the same $\langle\sigma_{(13)(24)}\rangle$-orbit. Applying specialization, we obtain the desired assertion.
\end{proof}

\subsection*{Burnside obstructions}

\begin{prop}
\label{prop:burn-ob-4}
The following $G$-actions on the following cubic threefolds are nonlinearizable,
for general values of parameters of the corresponding families: 
\begin{enumerate}
\item $2\sA_2+2\sA_1$, and $G=\langle\sigma_{(12)}, \sigma_{(12)(34)}\rangle$, from Proposition~\ref{prop:2A2+2A1class},  
\item $2\sA_3+2\sA_1$, $d(X)=2$, and $G=\langle\sigma_{(12)}, \sigma_{(12)(34)}\rangle$, from Proposition~\ref{prop:2A32A1},
\item $2\sD_4+2\sA_1$, and $G=\langle\sigma_{(12)}, \sigma_{(12)(34)}\rangle$, 
from Proposition~\ref{prop:2D4+2A1}, 
\item $4\sA_2$, and $G=\langle\sigma_{(12)}, \sigma_{(12)(34)}\rangle$ in the cases when $r_2=r_3$ from Proposition~\ref{prop:4A2-newform}, 
\item $4\sA_2$, and $G=\langle\eta,\sigma_{(234)},\sigma_{(12)(34)}\rangle\simeq C_3\times \fA_4$, from Proposition~\ref{prop:4A2-newform}. 
\end{enumerate}  
\end{prop}

\begin{proof}
In Cases (1)--(4), we are in the situation of Proposition~\ref{prop:D4Burn}: the involution $\sigma_{(12)}$ gives rise to a Burnside symbol of the form  
\begin{equation}
        (\langle \sigma_{(12)}\rangle, Y\actsfromleft k(S), (1)) \in \Burn_3(G),
\end{equation}
where $S\subset X$ is a cubic surface. The residual $Y$-action on $S$ 
fixes a smooth cubic curve, for general values of parameters, so that 
$$
\rH^1(Y, \Pic(\tilde{S}))=(\bZ/2)^2, 
$$
by \cite{BogPro},  
i.e., the symbol is incompressible. Moreover, linear actions do not contribute such symbols. 

In Case (5), we have an incompressible symbol
$$
(C_3, \fA_4\actsfromleft k(S'), (\zeta_3)), 
$$
where $S'$ is the Cayley cubic surface (unique cubic surface with 4 nodes). The $\fA_4$-action on $S'$ is birational to the linear $\fA_4$-action on $\bP^2$.
This symbol is incompressible, appears in the class $[X\actsfromright G]$ 
with multiplicity one, 
and distinguishes the given $G$-action on $X$ from a linear action, 
as in \cite[Remark 6.4]{CTZ}. 
\end{proof}

\section{Five Singular points }
\label{sect:five}

Let $X$ be a cubic threefold with five singular points. Under our assumptions, we need to consider the following combinations of singularities:
$$
2\sA_2+3\sA_1, 
\quad 2\sA_3+3\sA_1,\quad 3\sA_2+2\sA_1,\quad 5\sA_2, \quad  2\sD_4+3\sA_1.
$$
We adapt the argument in \cite[\S 6]{CTZcub}, which handles $5\sA_1$-singularities. First note that if the singularities are not in linearly general position, then there is a distinguished $G$-fixed singular point, and the $G$-action on $X$ is linearizable. Thus we can assume that the singular points of $X$ are 
$$
p_1=[1:0:0:0:0],\quad p_2=[0:1:0:0:0], \quad p_3=[0:0:1:0:0],
$$
$$p_4=[0:0:0:1:0],\quad  p_5=[0:0:0:0:1].$$
Automorphisms $\Aut(\bP^4,5)$ of $\bP^4$ respecting these points fit into the exact sequence:
$$
0\to\bG_m^4(k)\to \Aut(\bP^4,5)\stackrel{\rho}{\to}\fS_5\to 0.
$$

\begin{lemm}\label{lemm:5trans}
  Let $X$ be a cubic threefold with at most $\sA_n$-singularities and  $\Sing(X)=\{p_1,\ldots, p_5\}$. Let $G\subseteq \Aut(X)$ be a finite subgroup acting intransitively on $\Sing(X).$ Then the $G$-action on $X$ is linearizable.    
\end{lemm}
\begin{proof}
    If $G$ fixes a singular points, it is linearizable via projection. It suffices to show linearizability when $\rho(G)=C_2\times C_3$ or $C_2\times \fS_3$, i.e., when $G$ preserves the set $\{p_1, p_2\}$ and $\{p_3, p_4, p_5\}.$ In these cases, we can find an element $\sigma\in G$ such that $\rho(\sigma)=(1,2)(3,4,5)$. By conjugation under the torus action, one may assume that $\sigma$ permutes the coordinates $x_1,\ldots, x_5$ as the cycle $(1,2)(3,4,5)$. The $\langle\sigma\rangle$-invariant cubic threefolds with only $\sA_n$-singularities are given by 
    \begin{multline}\label{eq:5Anform}
        x_3x_4x_5+a(x_1x_2x_3 + x_1x_2x_4 + x_1x_2x_5)+\\
        b(
    x_1x_3x_4 + x_2x_3x_4 + x_1x_3x_5 + x_2x_3x_5 + x_1x_4x_5 + x_2x_4x_5)=0,
    \end{multline}
for some $a,b\in k^\times$ (via {\tt Magma}). Notice that if $a=0$, the cubic has nonisolated singularities, and if $b=0$, $p_1$ and $p_2$ are $\sD_4$-points. 
Based on the form 
    of \eqref{eq:5Anform}, one can see that the embedding $\bG^4_m(k)\subset G$ is trivial, i.e., $G$ does not contain diagonal elements. 
    
    If $\rho(G)=C_2\times \fS_3$, then there exists $\tau\in G$ such that $\rho(\tau)=(3,4)$ and $\tau$ had order $2.$ Thus $\tau$ takes the form
    $$
    \rho((x_1,x_2,x_3,x_4,x_5))=(a_1x_1,a_2x_2,x_4,x_3,a_3x_5),\quad a_1,a_2,a_3=\pm1.
    $$
    The only possibility for leaving \eqref{eq:5Anform} invariant is $a_1=a_2=a_3=1$. 
    
    Therefore, we conclude that $G=C_2\times C_3$ or $C_2\times\fS_3$, acting via corresponding permutations on the coordinates. Then $G$ pointwise fixes the line $l\subset \bP^4$ through $[1 : 1 : 0 : 0 : 0]$ and $[0 : 0 : 1 : 1 : 1]$. Let $\iota$ be the standard Cremona transformation on $\bP^4$. Observe that $\iota$ birationally transforms $X$ to a smooth quadric threefold $Q$, and $\iota(l)=l$. The intersection $l \cap Q $ contains $G$-fixed points, which implies the assertion.
\end{proof}

    \begin{coro}
    Let $X$ be a cubic threefold with singularities of type 
    $$
    2\sA_2+3\sA_1, \quad 2\sA_3+3\sA_1,\text{ or } \, \, 3\sA_2+2\sA_1.
    $$
    Then the $\Aut(X)$-action on $X$ is linearizable.
\end{coro}

Now we consider the cases when $G\subseteq \Aut(X)$ acts transitively on $\Sing(X)$: it follows that the only possible singularity type is $5\sA_2$.

\begin{lemm}
Let $X$ be a cubic threefold with $5\sA_2$-singularities such that $\Aut(X)$ does not fix any singular points. Then up to isomorphism, $X$ is given by 
\begin{multline}\label{eqn:5a3}
    x_1x_2x_3 + x_2x_3x_4 + x_1x_2x_5 + x_1x_4x_5 + x_3x_4x_5\\
+ a(x_1x_2x_4 + x_1x_3x_4 + x_1x_3x_5 + x_2x_3x_5 + x_2x_4x_5)=0,
\end{multline}
with $a=\zeta_3$ or $\zeta_3^2$, and $\Aut(X)=\fA_5$ is generated by 
\begin{align*}
    \sigma&:(\mathbf x)\mapsto (x_2,x_3,x_4,x_5,x_1),\\
      \tau&:(\mathbf x)\mapsto (a^2x_2,ax_1,a^2x_4,ax_3,x_5).
\end{align*}
\end{lemm}
\begin{proof}
By assumption, there exists an element $(1,2,3,4,5)\in\rho(G)$.  Up to conjugation by the torus action, we may assume that $G$ contains $\sigma$ as is given in the assertion. The condition that $X$ has 5$\sA_2$-singularities at $p_1,\ldots,p_5$ and that $X$ is left invariant by $\sigma$ forces $X$ to have equation \eqref{eqn:5a3}. 
By computation, one can check that $(1,2)\in\fS_5$ does not lift to $\Aut(X)$, and $G$ contains no torus action. On the other hand, $X$ is left invariant by $\tau$. It follows that $\Aut(X)=\fA_5$.
\end{proof}

In this case, the $\fA_5$-action on $X$ is not linearizable: the image $\iota(X)$ under the standard Cremona transformation is a smooth quadric threefold. This $\fA_5$-action on a smooth quadric is (conjecturally) not linearizable, by a work in progress \cite{antZ}.


\subsection*{Singularity type $2\sD_4+3\sA_1$}

Assume that $p_1, p_2$ are the $\sD_4$-points and $p_3, p_4, p_5$ the $\sA_1$-points. 
 Following the proof of Lemma~\ref{lemm:5trans},
 if $G$ does not fix a singular point, then $X$ is given by
$$
x_3x_4x_5=x_1x_2(x_3+x_4+x_5),
$$
with the $\fS_3$-action permuting $x_3,x_4,x_5$ and the infinite dihedral group $\fD_\infty=\bG_m(k)\rtimes C_2$ acting  on $x_1,x_2$. Applying the Cremona transformation $\iota$, based at the 5 singular points, we obtain the smooth quadric 
$$
x_1x_2=x_3x_4+x_4x_5+x_3x_5. 
$$
The $G$-action does not have fixed points, by our assumptions. 

We can apply the Burnside formalism. Consider
$$
G\simeq C_2^2 \times  \fS_3\subset \fD_\infty\times \fS_3, 
$$
where one generator of $C_2^2$ switches $x_1,x_2$ and 
the other multiplies $x_1,x_2$ by $-1$.
The first gives rise to the symbol
$$
(C_2, C_2\times \fS_3\actsfromleft k(S), (1)) \in \Burn_3(C_2^2),
$$
with residual action 
on the quadric $S$, given by $x_1^2=x_3x_4+x_4x_5+x_3x_5$. By \cite[Example 9.2]{TYZ-3}, this is an incompressible symbol, and the $G$-action on the quadric threefold is not linearizable.

\section{Six singularities}
\label{sect:six}

The relevant cases are
\begin{equation} 
\label{eqn:rela}
2\sA_2+4\sA_1,\quad 2\sA_3+4\sA_1.
\end{equation}
By Propositions \ref{prop:4A1+2A3} and \ref{prop:4A1+2A2}, we may assume that  the $4 \sA_1$-points are 
$$
p_1=[1:1:1:0:0],\qquad p_2=[-1:1:1:0:0],
$$
$$
p_3=[1:-1:1:0:0],\qquad p_4=[1:1:-1:0:0],
$$
and the two $\sA_2$ or $\sA_3$-points are 
$$
p_5=[0:0:0:1:0],\qquad p_6:=[0:0:0:0:1].
$$

\begin{prop} 
\label{prop:6}
Let $X$ be a cubic threefold with six singularities which are not all $\sA_1$-points. Then the $\Aut(X)$-action on $X$ is linearizable.    
\end{prop} 

\begin{proof} 
In both cases of \eqref{eqn:rela}, the four $\sA_1$-singularities are necessarily in a 
$G$-stable plane $\Pi\subset X$, and the two points with worse singularities define a $G$-stable line $l\subset X$, which is disjoint from $\Pi$. 
Arguing as in \cite[Lemma 1.1]{CTZ}, we obtain linearization.  
\end{proof} 

Normal forms in these two cases in \eqref{eqn:rela} are not needed for the study of linearizability. 
Nevertheless, we present them, for completeness.

\begin{prop}
Let $X$ be a cubic threefold singular at $p_1,\ldots,p_6.$
\begin{itemize}
    \item 
If $X$ has singularity type $2\sA_2+ 4\sA_1$, then $X$ is given by
\begin{multline}\label{eq:4A1+2A2eq}
(a_1x_4+a_2x_5)(x_1^2 - x_3^2)+(a_3x_4+
    a_4x_5)(
    x_2^2 - x_3^2)+\\+
    (a_5x_1+a_6x_2+a_7
    x_3)x_4x_5=0,
\end{multline}
for general $a_1,\ldots,a_7\in k$ satisfying
\begin{align}\label{eq:4A1+2A2condition}
&a_1^2a_6^2 + a_1a_3a_5^2 + a_1a_3a_6^2 - a_1a_3a_7^2 + a_3^2a_5^2=0,\\
    &a_2^2a_6^2 + a_2a_4a_5^2 + a_2a_4a_6^2 - a_2a_4a_7^2 + a_4^2a_5^2=0\nonumber.
\end{align}
 \item  If $X$ has singularity type $2\sA_3+4\sA_1$, then $X$ is given by
    \begin{align}
        \label{eq:2A3+4A1form}
    (x_1^2-x_3^2)x_4+  (x_2^2-x_3^2)x_5+x_3x_4x_5=0.
    \end{align}
    \end{itemize}
\end{prop}
\begin{proof}
    Singularities at $p_1,\ldots,p_6$ impose linear conditions on the vector space $\rH^0(\bP^4,\cO_\bP^4(3)).$ In particular, every cubic threefold singular at $p_1,\ldots,p_6$ is of the form ~\eqref{eq:4A1+2A2eq}, with general parameters $a_1,\ldots, a_7$. Assume that $p_5, p_6$ are $\sA_2$-points. This implies  that the quadratic terms locally at $x_4=1$ and $x_5=1$ define a degenerate quadratic form in four variables. This gives the nonlinear conditions \eqref{eq:4A1+2A2condition}. A general solution to this system of equations in $a_1,\ldots,a_7$  defines a cubic with $2\sA_2+4\sA_1$-singularities via ~\eqref{eq:4A1+2A2eq}.

    If $X$ has $2\sA_3+4\sA_1$-singularities, from Proposition~\ref{prop:4A1+2A3}, we know that  it contains the five planes spanned by points 
    $$
        \Pi_1\supset\{p_2,p_3,p_5\},\quad\Pi_2\supset\{p_1,p_4,p_5\},\quad\Pi_3\supset\{p_1,p_2,p_3,p_4\},
        $$
        $$
        \Pi_4\supset\{p_1,p_2,p_6\},\quad\Pi_5\supset\{p_3,p_4,p_6\}.
   $$
   This imposes further linear conditions $a_2=a_3=0.$ Substituting into \eqref{eq:4A1+2A2condition}, we also have $a_1a_6=a_4a_5=0$. When $a_1$ or $a_4=0$, the cubic will be reducible, thus $a_5=a_6=0.$ Moreover, $a_7\ne 0$, since otherwise $X$ has  nonisolated singularities. By scaling $x_4$ and $x_5,$ we obtain the form ~\eqref{eq:2A3+4A1form}.
\end{proof}

\bibliographystyle{plain}
\bibliography{sing-cube}

\end{document}